%% file: rationalgenusbounds-MTH-Hedden.tex
\documentclass[11pt]{amsart}
\font\emailfont=cmtt10

\usepackage[left=1.4 in, right=1.4 in, top=1 in, bottom=1 in]{geometry}

\usepackage{amsmath,amsthm,amsfonts,amscd,flafter,epsf}
\usepackage{graphicx}
\usepackage{epsfig}
\usepackage{psfrag}
\usepackage{subcaption}
\usepackage{mathtools}
\usepackage[table]{xcolor}
\usepackage[all]{xy}
\usepackage[color=blue!20!white,textsize=tiny]{todonotes}
\usepackage{tikz}
\usepackage{thmtools}
\usepackage{thm-restate}
\usepackage{hyperref}

\input{macros-MTH-Hedden}

\title{A 4-dimensional rational genus bound}

\author[Matthew Hedden]{Matthew Hedden}
\address{Department of
Mathematics, Michigan State University \newline
\indent{\emailfont{mhedden@math.msu.edu}}}

\author{Katherine Raoux}
\address{Department of Mathematical Sciences, University of Arkansas \newline\indent{\emailfont{kraoux@uark.edu}}}

\thanks{MH gratefully acknowledges support from  DMS-1709016 and DMS-2104664.  KR was partially supported by an AWM Mentoring Travel Grant and the Max Planck Institute for Mathematics.}

\begin{document}
\maketitle

\begin{abstract}
    We  introduce a 4-dimensional analogue of the rational Seifert genus of a knot $K\subset Y$, which we call the \emph{rational slice genus}, that measures the complexity of a  homology class in $H_2(Y\times [0,1],K;\Q)$.  Our main theorem is a lower bound for the rational slice genus of a knot in terms of its Heegaard Floer $\tau$ invariants. To prove this, we  bound  the $\tau$ invariants of any satellite link whose pattern is a closed braid  in terms of the $\tau$ invariants of the companion knot, a result which should be of independent value. Our techniques also produce rational PL slice genus bounds. 

    As applications, we use our bounds to prove that Floer simple knots have rational slice genus equal to their rational Seifert genus. We also show that there exist sequences of knots in a fixed 3-manifold whose PL slice genus is unbounded. In addition, we produce stronger bounds for the slice genus of knots relative to the rational longitude, and use these to produce a rational slice-Bennequin bound for knots in contact manifolds with non-trivial contact invariant.
\end{abstract}

\section{Introduction} 
Minimal genus problems play a central role in low-dimensional topology, exemplified in dimension 4 by the Thom and Milnor Conjectures and, in dimension 3, by the importance of Thurston's semi-norm. At their core, these problems typically ask: what is the minimum genus of an embedded, oriented surface representing a given {\em integral} class in $H_2(M;\Z)$?  For instance, while the Thurston norm is a function on $H_2(Y^3;\R)$, its linearity ensures it is determined by its values on integral classes.  In the context of knot theory, the Seifert and slice genera of a knot are the minimum genera of  embedded surfaces representing classes in $H_2(S^3,K;\Z)$ or $H_2(S^3\times[0,1],K;\Z)$, respectively, which map to the fundamental class $[K]\in H_1(K)$ under the connecting homomorphism.

We find one notable exception in \cite{Turaev-function}, where Turaev considers an analogue of the Thurston norm which applies to classes in $H_2(Y;\Q/\Z)$ that are {\em not} in the image of $H_2(Y;\Q)$. At best, such classes can only be represented by \emph{folded surfaces}. These are oriented surfaces that are smoothly embedded away from a singular set, which we can always take to be homeomorphic to a single circle. 
Thus, the singular set is a knot in the 3-manifold $Y$ which represents the (non-zero) image of the class of the folded surface under the Bockstein homomorphism: $H_2(Y;\Q\slash\Z)\to H_1(Y\;\Z)$. 

A folded surface whose singular set is a fixed knot $K$ is also known in the literature as a \emph{rational Seifert surface} for $K$. From this perspective, we consider it as a map of an oriented surface with boundary $\varphi:(F,\partial F)\to (Y,K)$ whose restriction to the interior of $F$ is an embedding and to the boundary is a covering map, see Definition \ref{defn:rationalSeifert} (cf. \cite{Calegari-Gordon,Ni-Wu}). 

These considerations lead to two minimal genus questions: 
\begin{enumerate}
\item For a fixed knot $K\subset Y$, what is its rational Seifert genus, $\lVert K\rVert_Y$?
\item What is the infimum of the rational genus over knots in a fixed homology class?
\end{enumerate} 
Here, the \emph{rational (Seifert) genus} $\lVert K\rVert_Y$ is the infimum of $-\chi(F)\slash 2p$ where $F$ is any rational Seifert surface with no sphere or disk components and $p$ is the index of the covering of $K$ on the boundary, see \cite[Definition 2.3]{Calegari-Gordon}. Calegari and Gordon compute the rational genus in terms of the Thurston norm of a class in the relative homology of the knot complement. It follows that the infimum is  achieved, takes rational values, and is determined once again by its value on a particular integral homology class. Answering the second question is equivalent to a calculation of  Turaev's complexity function for a particular homology class. 

The purpose of this article is to study a 4-dimensional analogue of the first question. Our main result is a lower bound for the \emph{rational slice genus} $\lVert K\rVert_{Y\times[0,1]}$ in terms of the breadth of the Heegaard Floer $\tau$ invariants of the knot:
\begin{restatable}{thm}{RationalGenusBound}\label{thm:RationalGenusBound}
For any rationally null-homologous knot $K$ in a 3-manifold $Y$, $$\tau_{\max}(K)-\tau_{\min}(K)\leq 2\lVert K \rVert_{Y\times [0,1]}+1.$$
\end{restatable}
\noindent Here, $\tau_{\max}(K)$ and $\tau_{\min}(K)$ are the maximum and minimum, respectively, of the invariants $\tau_{\alpha}(K)$ taken as $\alpha$ ranges over all nontrivial Floer classes in $\HFa(Y)$. These invariants, defined in the authors' previous work \cite{Raoux, RelativeAdjunction}, extend the Ozsv\'ath-Szab\'o concordance invariant $\tau$ \cite{Knots,RasThesis} to other 3-manifolds, and to rationally null-homologous links therein. The rational slice genus appearing in the inequality is defined using rational slice surfaces in direct analogy with their 3-dimensional counterparts:
\begin{restatable}{defn}{RationalSliceSurface}\label{defn:RationalSliceSurface}
A \emph{rational slice surface} for a knot $K\subset Y$ is a compact, oriented surface $\Sigma$ with boundary, along with a smooth map $\varphi:(\Sigma,\partial\Sigma)\to (Y\times [0,1], K\times \{1\})$ such that $\varphi\vert_{int(\Sigma)}$ is an embedding in $Y\times [0,1]\setminus K\times \{1\}$ and such that $\varphi\vert_{\partial\Sigma}$ is an oriented covering map $\partial\Sigma\to K\subset Y\times \{1\}$ of degree $p$. We use $\Sigma$ to denote the singular (folded) surface arising as the image of this map. 
The \emph{rational slice genus} is:
$$\lVert K\rVert_{Y\times [0,1]}=\inf_{\Sigma,p}\frac{\chi^-(\Sigma)}{2p}$$ where $\Sigma$ has no closed components, $\chi^-(\Sigma)=\max\{-\chi(\Sigma),0\}$ and the infimum is taken over all rational slice surfaces $\Sigma$ and $p$, 
\end{restatable}

The proof of Theorem 1 has three main steps. First, embedded on a rational slice surface in a neighborhood of the singular set $K$ is a null-homologous satellite link, $P_\beta(K)$. In contrast to the 3-dimensional setting, where this link is a \emph{specific} null-homologous cable \cite{NiThurston, Calegari-Gordon}, in dimension 4 the pattern can be any closed $p$-braid, where $p$ is the order of the cover $\partial\Sigma\to K$, see Section \ref{sec:RatSlice}. Nevertheless, since $P_\beta(K)$ is null-homologous, the main result of \cite{RelativeAdjunction} now implies the genus of the rational slice surface is bounded below by $\tau_{\alpha}(P_{\beta}(K))$. 

For the second step, we relate $\tau_{\alpha}(P_{\beta}(K))$ to $\tau_{\alpha}(K)$: 

\begin{restatable}{thm}{BPBounds}\label{thm:BPBounds} Let $K\subset Y$ be a rationally null-homologous knot and $P_{\beta}(K,\lambda)$ the satellite of $K$ whose pattern is the closure of a $p$-braid $\beta$ formed with respect to the framing $\lambda$ for $K$. Then for any $\alpha\in\HFa(Y)$ and $\Theta\in\HFa(\#^{\lvert \beta \rvert-1}S^1\times S^2)$
$$\lvert 2\tau_{\alpha\otimes \Theta}(P_{\beta}(K,\lambda))-2p\tau_{\alpha}(K)-(p-1)p\selflink(K,\lambda)-\omega(\beta)\rvert \leq (p-1)+|\beta|-1,$$
where $\lvert \beta \rvert$ is the number of components of the closure of $\beta$, $\writhe(\beta)$ its writhe, and $\selflink(K,\lambda)$ the rational linking of $K$ with $\lambda$ (see Section \ref{sec:RatLink}).
\end{restatable}
\noindent When $P_{\beta}(K)$ is a link, the invariants $\tau_{\alpha\otimes\Theta}$ are defined in terms of its knotification, which is a knot in $Y\#^{|\beta|-1}\SoneStwo$. See Definition \ref{defn:linktau}.

The bound in Theorem \ref{thm:BPBounds} may be of independent interest, since it applies to \emph{any} knot (including null-homologous knots) in \emph{any} 3-manifold and {\em any}  braided pattern -- not only those for whom the corresponding satellites are null-homologous. Techniques such as bordered Floer theory might eventually allow one to compute $\tau_{\alpha\otimes \Theta}(P_{\beta}(K,\lambda))$ explicitly for those  braids whose closures in $S^1\times D^2$ have computable filtered $A_\infty$ modules, analogous to the special case of cable knots \cite{Hom-cables} cf. \cite{CableII,LOT,Petkova,Chen}. Even for such braids, however, current bordered technology requires the underlying 3-manifold to be an $L$-space, whereas our bound applies in general.

For the final step, we observe that steps one and two give us a family of genus bounds. Indeed, we obtain a bound for every possible braided satellite that could come from a rational slice surface and for each non-zero Floer homology class. Theorem \ref{thm:RationalGenusBound} then follows by taking an infimum over this family of bounds.

\subsection{Applications} One motivation for studying rational genus stems from its relevance to the Berge Conjecture \cite[Problem 1.78]{KirbyProblems} (see also \cite{Berge}), which is equivalent to the assertion that any knot in a lens space admitting a 3-sphere surgery is {\em simple}. In \cite{BGH},  Baker, Grigsby and the first author outlined a Floer theoretic approach: first show that any knot admitting a 3-sphere surgery is {\em Floer simple}, and then show that the only Floer simple knots in lens spaces are the simple knots; that is, simple knots are characterized by their (simple) Floer homology (Floer simple means that $\rank\ \HFKa(Y,K)= \rank\ \HFa(Y)$). The first part of their approach was carried out in \cite{Hedden-Berge}, and independently by Rasmussen \cite{Rasmussen-Berge}.  Rasmussen's work moreover suggested a geometric route towards the Floer homological characterization of simple knots. He showed that knots in $L$-spaces admitting 3-sphere surgeries minimize rational genus amongst all knots in their homology class, and asked whether this minimality feature distinguishes Floer simple knots. He showed, in particular, that the Berge conjecture would be resolved by a positive answer to the following: 
\begin{question}\cite{Rasmussen-Berge}\label{ques:Unique}
Is every simple knot with $\lVert K\rVert_{L(p,q)}< 1/2$ the \emph{unique} genus minimizer in its homology class?  
\end{question}
\noindent  Progress has been made along these lines.  Ni and Wu  established that {\em all} Floer simple knots in {\em all} $L$-spaces are genus minimizers in their homology class \cite{Ni-Wu}. On the other hand, Greene and Ni gave examples of non-simple genus minimizers in lens spaces with $\lVert K\rVert_{L(p,q)}> 1/2$ \cite{Greene-Ni}. Question \ref{ques:Unique} remains unresolved, however, and motivates the general search for geometric features of Floer simple knots that could be used to characterize them. To this end, we use  Theorem \ref{thm:RationalGenusBound} to show that the rational Seifert and slice genera of Floer simple knots are equal:

\begin{restatable}{thm}{FSK}\label{thm:FSK}
If $K$ is a Floer simple knot, then $$\lVert K\rVert_Y=\lVert K\rVert_{Y\times [0,1]}=\frac{1}{2}(\tau_{\max}(K)-\tau_{\min}(K)-1).$$
\end{restatable} 

 \noindent Simple knots with $3$-sphere surgeries  are  rationally fibered, by Ni \cite{NiFibered}, and  work of Plamenevskaya and the first author shows they induce contact structures with nontrivial Ozsv\'ath-Szab\'o contact invariant \cite{rationalcontact}. Hence, we might expect, in analogy with strongly quasipositive null-homologous knots in the 3-sphere, that they have rational Seifert genus equal to their rational slice genus. Theorem \ref{thm:FSK} establishes that, in fact, all Floer simple knots have this property.  Given this, it is natural to wonder whether their rational slice genus, like their rational Seifert genus, is minimal amongst all knots in their homology class.  

\begin{question}
    Do Floer simple knots minimize rational slice genus amongst all knots in their homology class? 
\end{question}

\noindent This question should be viewed from the perspective of a 4-dimensional version of Turaev's function \cite{Turaev-function}, which we introduce in Subsection \ref{sec:RatGenus}, and where we pose several related questions. \\

In another direction, one can use our results to study minimum genus problems for PL surfaces in smooth 4-manifolds with boundary.  Indeed, while our initial motivation for studying  $\tau_{\max}(K)-\tau_{\min}(K)$ was Theorem \ref{thm:RationalGenusBound} and its application to Floer simple knots, the difference function in fact bounds the much subtler  rational PL slice genus.  

\begin{restatable}{thm}{rationalPLslice}\label{cor:rationalPLslice}
     Let $\lVert K\rVert^{PL}_{Y\times [0,1]}$ denote the rational PL slice genus; i.e., the analogue of rational slice genus where the maps of surfaces to $Y\times[0,1]$ are  piecewise linear.  Then 
    $$\tau_{\max}(K)-\tau_{\min}(K)\leq 2\lVert K \rVert^{PL}_{Y\times [0,1]}+1.$$
\end{restatable}

\noindent In particular, we obtain bounds on the rational PL slice genus of {\em null-homologous} knots.  Using these in conjunction with our bounds for the $\tau$ invariants of braided satellites, we can easily produce families of knots  whose rational PL slice genus, hence PL slice genus, is unbounded. For example:

\begin{restatable}{corollary}{reproof}\label{cor:reproof} Let $K$ be any $L$-space knot except the trefoil. Then for any $i$ there exists a  knot $J_i\subset S_{-1}^3(K)$ for which the genus of any properly embedded PL surface $\Sigma\subset S^3_{-1}(K)\times[0,1]$ bounded by $J_i$ must be bigger than $i$.
    
\end{restatable}

These applications were motivated by recent work of Hom, Stoffregen and Zhou, who produced a sequence of knots $K_n$ in homology spheres $Y_n$, the latter of which  bound  homology balls $W_n$.  Their main theorem \cite[Theorem 1.1]{hom2023plgenus} shows that the minimum genus of any embedded PL surface bounded by $K_n$ in $W_n$ is increasing (hence unbounded) in $n$.  The natural homology balls $W_n$ bounded by $Y_n$ are obtained from $S^3_{-1}(T_{2n,2n+1})\times[0,1]$, $n\ge0$ by removing a neighborhood of an arc connecting its boundary components.  Letting $K=T_{2n,2n+1}$ in Corollary \ref{cor:reproof}, we obtain the seemingly stronger result that for any {\em fixed}  $Y_n$, $n>0$, there exists a sequence of knots $\{J_i\subset Y_n\}_{i>0}$ for whom the PL slice genus in $W_n$ is unbounded.  On the other hand, we can only bound the genera of PL surfaces in the particular homology ball $W_n$ whereas \cite[Theorem 1.1]{hom2023plgenus} applies to  any homology ball bounded by $Y_n$.

We also emphasize that in the null-homologous case the difference in $\tau$ invariants provides a lower bound for the {\em rational} PL slice genus.  One might expect that for null-homologous knots the rational (PL) slice genus would coincide with the (PL) slice genus.  This is not the case, and indeed the rational slice genus can be strictly smaller than the slice genus, even for knots in the 3-sphere.  See Example \ref{ex:SmallGenusCable}.  This raises a natural question:

\begin{question}\label{question:slice}
    Are there non-slice knots in the 3-sphere with $\lVert K\rVert_{S^3\times [0,1]}=0$? Is there a sequence of non-slice knots $\{K_n\}$ with  $\underset{n\rightarrow\infty}{lim} \lVert K_n\rVert_{S^3\times [0,1]}=0 $?
\end{question}

This question could be asked for knots in any 3-manifold, but the classical setting underscores the distinction between the slice genus and its rational analogue.  The latter should be viewed as the geometric  complexity of the rational homology class in $H_2(Y\times[0,1],K;\Q)$ which maps to $[K]\in H_1(K;\Q)\cong \Q$, whereas the former is the complexity of the corresponding integral class (which exists only when $K$ is null-homologous).

One could also study the {\em topological} rational slice genus, defined with surfaces that are locally flatly embedded away from their boundary. The behavior of the topological 4-genus under satellite operations is much different than the smooth 4-genus.  For instance, \cite[Example 1.3]{MR4464463} shows that the topological 4-genus of the $(n,1)$ cable of the trefoil is $1$, from which it follows that its  topological rational slice genus is zero.  Thus the answer to Question \ref{question:slice} in the topological category is affirmative. \\

As an additional application of the methods at hand, we produce examples of \emph{deep slice knots} in the boundary of certain rational homology balls. A knot in the boundary of a 4-manifold $X$ is slice in $X$ if it bounds a properly embedded disk. Such a slice disk is considered \emph{deep} if it is not contained in a collar neighborhood of the boundary. In \cite{Klug-Ruppik}, Klug and Ruppik reformulate Kirby Problem 1.95 (attributed to Akbulut) to ask: are there smooth contractible 4-manifolds with integer homology sphere boundary and which contain deep slice knots that are null-homotopic in the boundary?

We do not address that question directly. Nevertheless, we observe the $\tau_\alpha$ invariants defined in \cite{RelativeAdjunction} are deep slice obstructions. We produce several examples of deep slice knots in rational homology spheres:
\begin{restatable}{proposition}{DeepSlice}\label{prop:DeepSlice}
The lifts of the knots $8_{20}$, $10_{129}$ and $10_{140}$ are each deep slice in the branched double cover over their respective slice disks.
\end{restatable} 
\noindent To produce examples, we start with a slice knot in the 3-sphere and consider its lift in the double branch cover. We use data collected by Levine \cite{ALevine-branchedcovers} to compute certain $\tau$ invariants of the lifted knot. We do not currently know whether these examples are null-homotopic in the boundary, but expect they are not. 

\subsection{A related 4-dimensional genus problem} Wu and Yang recently studied a related notion of slice genus for rationally null-homologous knots, which one can regard as the slice genus \emph{relative to the rational longitude} \cite{Wu-Yang}.   Concretely, the quantity they measure is the minimum of $\frac{\chi^-(\Sigma)}{2p}$, where $\Sigma$ is any surface in $Y\times [0,1]$ whose boundary maps to a specific cable link $K_{mr,ms}$, where $p=mr$ is the order of $[K]$, and the relatively prime pair $(r,s)$ is determined by the linking form on $\Tors(H_1(Y))$ applied to $[K]$.  The cable knot $K_{r,s}$, embedded on the boundary of $\nu(K)$ is often called the {\em rational longitude}.  Thus one can interpret their minimal genus problem for rationally null-homologous knots as the traditional slice genus problem for the cable link $K_{mr,ms}$.  

In contrast, the rational slice genus we study is, ultimately, an infimum taken over a doubly infinite collection of minimal genus problems:  we consider the ordinary slice genus of all null-homologous braided satellites of $K$, taking the infimum over all braids $\beta\in B_n$ in all braid groups whose indices $n$ are multiples of the order of $[K]$:

$$\lVert K\rVert_{Y\times [0,1]}=\inf_{n\ge 1}\left(\inf_{\beta\in B_n}\left(\min_{\Sigma}\bigg\{\frac{\chi^-(\Sigma)}{2n}\ \bigg|\ \Sigma\hookrightarrow Y\times[0,1]\ \mathrm{s.t.} \ \partial\Sigma= P_\beta(K)\bigg\}\right)\right). $$

\noindent The  slice genus relative to the rational longitude considered in \cite{Wu-Yang} corresponds to the case that $n=mr$ equals the order of $[K]\in H_1(Y)$ and $\beta\in B_{mr}$ is the standard braid presentation of the $(mr,ms)$ torus link.   Thus the rational slice genus is clearly bounded above by the slice genus relative to the rational longitude.   

Our techniques readily show that the slice genus relative to the rational longitude is bounded below by twice the value of {\em any} of the $\tau_\alpha$  invariants, see Corollary \ref{cor:RatLongframed}.  Since the difference $\tau_{\max}-\tau_{\min}$ is typically smaller than $2\tau_{\max}$, we expect that the rational slice genus will usually be strictly smaller than the slice genus relative to the rational longitude.  Example \ref{ex:SmallGenusCable} provides a specific instance where this happens.  It would be interesting to better understand the class of knots for which the two notions agree (note that Floer simple knots are among them, according to Theorem \ref{thm:FSK}).  We can also use our bound for the slice genus relative to the rational longitude  to prove a slice-Bennequin inequality for rationally null-homologous knots:

 \begin{restatable}{thm}{rationalslicebennequin}\label{thm:rationalslicebennequin}
Let $(Y,\xi)$ be a contact 3-manifold with non-trivial Ozsv\'ath-Szab\'o contact class. If $\mathcal{K}$ is a rationally nullhomologous Legendrian knot, then \[\tb_{\Q}(\mathcal{K})+\rot_{\Q, [\Sigma]}(\mathcal{K}) \leq -\frac{1}{p}\chi(\Sigma)\] where $\Sigma\hookrightarrow Y\times[0,1]$ is a rational slice surface of degree $p$ over $\mathcal{K}$ whose induced satellite is a multiple of the rational longitude.
\end{restatable}
\noindent Here $\tb_{\Q}(\mathcal{K})$ and $\rot_{\Q,[\Sigma]}(\mathcal{K})$ are rational versions of the Thurston Bennequin and rotation numbers defined by Baker and Etnyre in \cite{Baker-Etnyre}. The boundary condition is necessary for $\tb_{\Q}(\mathcal{K})$ to be well-defined, see Section \ref{sec:sliceTB}.

\subsection{Outline} The paper is structured as follows. Section 2 contains background material on rational linking and rational genus. It also contains a brief overview of knot Floer homology and its $\tau$ invariants, and connects these concepts to rational linking numbers. In Section 3, we compute bounds for the $\tau$ invariants of $(p,pn+1)$ cables of rationally null-homologous knots. Section 4 contains a generalization of the relative adjunction inequality of \cite{RelativeAdjunction} to link cobordisms and the proof of Theorem \ref{thm:BPBounds}. We prove Theorem \ref{thm:RationalGenusBound} in Section 5. Finally, Section 6 contains applications including the proofs of Theorem \ref{thm:FSK} and Corollary \ref{cor:reproof}.

\subsection{Acknowledgments} The authors thank Chuck Livingston, Quiyu Ren, Daniel Ruberman, Zhongtao Wu and Jingling Yang for helpful discussions, IMSA/University of Miami for the workshop ``Gauge Theory and Low Dimensional Topology" it held  in April 2023, and Jen Hom for her talk at the workshop, which inspired Corollary \ref{cor:reproof}.

\section{Background}

In this section, we summarize facts about rationally null-homologous knots, rational Seifert surfaces and $\tau$ invariants. 

\subsection{Rational linking numbers} \label{sec:RatLink}
A pair of disjoint rationally null-homologous knots $K_1$ and $K_2$ in a 3-manifold $Y$ have a well-defined rational valued linking number, defined by considering the homology class of one in the complement of the other. First observe that the meridian of $K_1$ bounds a disk $D_\mu$ in its tubular neighborhood. We have the following portion from the long exact sequence of the pair $(Y, Y-\nu(K_1))$:
$$0\to \Q\langle D_\mu\rangle \xrightarrow{\partial} H_1(Y-\nu(K_1);\Q)\to H_1(Y;\Q)\to 0,$$ where we use excision to identify $H_2(Y, Y-\nu(K_1);\Q)$ with $\Q\langle D_\mu\rangle$, and the fact that $K_1$ is rationally null-homologous to show that the map  $H_2(Y)\rightarrow H_2(Y,Y-\nu(K_1))$ is rationally trivial.  It follows that $$H_1(Y-\nu(K_1);\Q)\cong H_1(Y;\Q)\oplus \Q\langle \mu \rangle.$$ 
Since $K_2$ is also rationally null-homologous in $Y$, its class maps to a rational  multiple of the meridian under the splitting above. Thus, $[K_2]=\frac{r}{d}[\mu]$ for some rational number $\frac{r}{d}$. Define $$\lk_{\Q}(K_1,K_2):= \frac{r}{d}.$$
See also \cite[Exercise 4.5.12(c)]{Gompf-Stipsicz}

\subsection{Rational Seifert surfaces and linking}
Every rationally null-homologous knot bounds a singular surface called a rational Seifert surface, whose definition we  recall from \cite{Calegari-Gordon}, cf. \cite{Baker-Etnyre}.

\begin{defn}\label{defn:rationalSeifert}
A \emph{rational Seifert surface} for an oriented knot $K$ in a 3-manifold $Y$ of order $p$ is a compact, oriented surface $F$ with boundary, and a map $\varphi:(F,\partial F)\to (Y,K)$ such that $\varphi\vert_{int(F)}$ is an embedding in $Y\setminus K$ and $\varphi|_{\partial F}$ is an oriented covering map of degree $p$ over $K$. We use $F$ to denote the singular (folded)  surface arising as the image of this defining map.
\end{defn}

It is useful for calculations to reinterpret the rational linking number of a pair of rationally null-homologous knots as the intersection number of one knot with a rational Seifert surface for the other. To make this precise, let $c$ be a rational 2-chain with boundary $K_1$. Then $$\lk_\Q(K_1,K_2)=c\cdot K_2.$$ This is well-defined, since if $c'$ is a different 2-chain with boundary $K_1$, the difference $c-c'$ is a cycle, and therefore a well-defined class in $H_2(Y;\Q)$. At the same time, $[K_2]=0$ in $H_1(Y;\Q)$ thus the intersection $(c-c')\cdot K_2=0$. 

A rational Seifert surface $F_1$ for the knot $K_1$ gives us an explicit rational 2-chain $c=\frac{1}{p}F_1$ for the present purpose.  We then have $$\lk_\Q(K_1,K_2)=\frac{1}{p}F_1\cdot K_2.$$ 

\begin{prop}\label{prop:RatSS}
Let $F$ be a rational Seifert surface for an oriented knot $K$ in $Y$ that intersects the boundary of a tubular neighborhood of $K$ in a curve homologous to $p\lambda+r\mu$ for some choice of framing $\lambda$ for $K$. Then 
\begin{enumerate}
\item $\lk_\Q(K,\lambda)= -\frac{r}{p}$
\item $\lk_{\Q}(K,\lambda+n\mu)=\lk_{\Q}(K,\lambda)+n.$
\end{enumerate}
\end{prop}

\begin{proof}
 Consider a tubular neighborhood of $K$ in $Y$. Since $K$ is oriented, $\lambda$ and $\mu$ inherit an orientation from $K$. An oriented basis for the tangent bundle to $Y-\nu(K)$ along its boundary is $(\lambda,\mu,\eta)$ where $\eta$ is the outward pointing normal vector, and we conflate the curves $\lambda$ and $\mu$ with their tangent vectors. With this choice of orientation, we have $\lambda\cdot \mu=1$ and $\lambda\cdot\mu=-\mu\cdot\lambda$ in $\partial\nu(K)$. Thus, $\frac{1}{p}(p\lambda+r\mu)\cdot \lambda=-\frac{r}{p}$. 
 
 Similarly $\lk_{\Q}(K,\lambda+n\mu) =\frac{1}{p} (p\lambda+r\mu)\cdot (\lambda+n\mu)=\frac{pn}{p}\lambda\cdot \mu+\frac{r}{p}\mu\cdot \lambda=\frac{pn-r}{p}$.
\end{proof}

The rational linking number descends modulo $\Z$ to the well-known $\Q\slash\Z$-valued linking pairing on  torsion homology classes in $H_1(Y;\Z)$:
$$\lk_{\Q\slash\Z}:\Tor (H_1(Y;\Z))\times\Tor(H_1(Y;\Z))\to \Q\slash\Z$$ where $\lk_{\Q\slash\Z}([K_1],[K_2]):= \lk_{\Q}(K_1,K_2)\mod\Z$. Choosing different representatives of $[K_i]$ (such as different framings of $K_i$) may change $\lk_{\Q}$ by an integer, therefore the linking pairing is only well-defined modulo $\Z$. From here, we define the $\Q\slash\Z$ self-linking number of a knot:

\begin{defn}
The \emph{$\Q\slash\Z$ self-linking number} of a rationally nullhomologous knot $K$ in $Y$ is $$\qzselflink(K):=\lk_{\Q\slash\Z}([K],[K])=\lk_{\Q}(K,\lambda)\mod \Z$$ where $\lambda$ is any choice of framing for $K$.
\end{defn}
The above discussion shows the $\Q\slash\Z$ self-linking number is not only a knot invariant, but an invariant of the underlying \emph{homology} class of $K$. Consequently, any pair of cobordant knots in $Y\times [0,1]$ have the same $\Q\slash\Z$ self-linking number. 

\subsection{Rational slice surfaces}\label{sec:RatSlice}  Our goal is to study the following 4-dimensional analogue of rational Seifert surfaces, which we defined in the introduction and restate here for the reader.

\begin{defn}
A \emph{rational slice surface} for a knot $K\subset Y$ is a compact, oriented surface $\Sigma$ with boundary, along with a map $\varphi:(\Sigma,\partial\Sigma)\to (Y\times [0,1], K\times \{1\})$ such that $\varphi\vert_{int(\Sigma)}$ is an embedding in $Y\times [0,1]\setminus K\times \{1\}$ and such that $\varphi\vert_{\partial\Sigma}$ is an oriented covering map $\partial\Sigma\to K\subset Y\times \{1\}$ of degree $p$. We use $\Sigma$ to denote the singular surface arising as the image of this map. 
\end{defn}

Classically, a slice surface for a knot $K$ in the 3-sphere is a surface in the 4-ball with boundary $K$. We can always construct a slice surface from a Seifert surface, by pushing the Seifert surface into the 4-ball.  Similarly, we can push a rational Seifert surface $F$ in $Y$ into $Y\times [0,1]$ to get a rational slice surface. In contrast to the classical case, however,  near the singular set  a general rational slice surface may not look like the push-in of rational Seifert surface.  Rather than circles nearby the boundary mapping to a particular cable link of $K$, as they do for a rational Seifert surface, they can map to an arbitrary closed braid embedded in a solid torus neighborhood of $K$.

To clarify this important point, note that in $Y\times [0,1]$ the knot $K\times \{1\}\subset Y\times \{1\}$ has a simultaneously tubular and collar neighborhood:
\begin{equation}\label{eq:tubularcollar}
    \nu(K)\cong S^1\times D^2\times [1-\epsilon,1],
\end{equation}in which $K$ is embedded as $S^1\times \{ 0\}\times \{1\}$.  In the following, we regard rational slice surfaces up to isotopy of their defining maps which are fixed on the boundary.    
\
\begin{prop}\label{prop:neighborhood}
Let $\Sigma\subset Y\times [0,1]$ be a rational slice surface for a knot $K$ in $Y\times \{1\}$. Then for all $\eta\in [1-\epsilon,1)$ sufficiently close to 1, $\Sigma\cap Y\times\{\eta\}$ is a closed braid $\beta\subset S^1\times D^2\times \{\eta\}$   whose isotopy class in the solid torus is independent of $\eta$.  In particular, $\Sigma$ specifies a well-defined braided satellite link $P_\beta(K)$ with respect to the framing given by Equation \eqref{eq:tubularcollar}.  Any braided satellite $P_\beta(K)$ whose index is a multiple of the order of $[K]\in H_1(Y)$ arises from a rational slice surface in this manner. 
\end{prop}

\begin{proof} Up to isotopy of $\Sigma$ rel boundary, we may assume that $\partial\Sigma$ possesses  a collar neighborhood  $\partial \Sigma\times[1-\epsilon,1]$ which is mapped to $Y\times [1-\epsilon,1]$ in a level-preserving fashion.  
Since $\Sigma$ is embedded away from its boundary, this implies that  for $\eta$ near enough to $1$,   $\partial \Sigma\times \{\eta\}$ is mapped to a link $L\subset S^1\times D^2\times \{\eta\}$ embedded in the solid torus neighborhood of $K\times\{\eta\}\subset Y\times \{\eta\}$ provided by \eqref{eq:tubularcollar}.  Moreover, the fact that $\partial \Sigma\times\{1\}\rightarrow K\times\{1\}$ is a smooth covering implies that the projection $L\subset S^1\times D^2\times \{\eta\}\rightarrow S^1$ is also a covering for all $\eta$ sufficiently close to $1$.  Thus $L$ is none other than a closed braid, embedded in a solid torus neighborhood of $K$.  The level-preserving embedding of $\partial\Sigma\times [\eta,\eta']$ into $S^1\times D^2\times [\eta,\eta']$ provides an isotopy  in the solid torus between any two such closed braids, again provided $\eta$ is sufficiently close to $1$.   

It is straightforward to see that any closed braid $\beta$, embedded as a satellite $P_\beta(K)$  of $K$ in its tubular neighborhood, arises from some rational slice surface, provided only that the index of the braid is a multiple of the order of the homology class of $K$.  Indeed, such a closed braid satellite will be null-homologous and hence bound smoothly embedded slice-surfaces in $Y\times [0,1]$.  Extending such a surface by the mapping cylinder of the projection map from the solid torus to $K$, we arrive at a rational slice surface.
\end{proof}  

Throughout, we refer to the link $P_\beta(K)$  as the braided satellite \emph{induced by $\Sigma$}.

\subsection{Rational genera}\label{rationalgenusdefinition}
First, a bit of notation. For a connected, oriented surface $S$ write $\chi^-(S)=\max\{-\chi(S),0\}$. If $S$ is not connected, $\chi^-(S)$ is the sum of $\chi(S_i)$ over all connected components $S_i$ of $S$. 

Now for a rationally null-homologous knot $K$ in a 3-manifold $Y$ we define the \emph{rational Seifert genus} of $K$ to be: $$\lVert K\rVert_Y=\inf_{F}\frac{\chi^-(F)}{2p}$$ where the infimum is taken over all rational Seifert surfaces for $K$ and all $p$. 

In a similar manner, define the \emph{rational slice genus} of $K$ in $Y\times [0,1]$ to be: $$\lVert K\rVert_{Y\times [0,1]}=\inf_{\Sigma}\frac{\chi^-(\Sigma)}{2p}$$ where $\Sigma$ is a rational slice surface for $K$. Clearly, we have the inequality 
\begin{equation}\label{eqn:4genus3genus}
\lVert K\rVert_{Y\times [0,1]}\leq \lVert K\rVert_{Y}.
\end{equation}

\begin{remark}The rational slice genus is constant on concordance classes of knots in $Y\times [0,1]$. To see this, let $K$ be a knot in $Y\times \{1/2\}$ and $J$ a concordant knot in $Y\times\{1\}$. Now, if $\Sigma_K\subset Y\times [0,1/2]$ is a rational slice surface for $K$, we can construct a rational slice surface for $J$ as follows. Remove a neighborhood of the singular set of $\Sigma_K$. The resulting surface $\overline\Sigma$ now has boundary that is a satellite $P_{\beta}(K,\lambda)$ of $K$. Extend $\overline\Sigma$ by the oft-used induced concordance to the corresponding satellite of $J$. This yields a surface, of the same topological type, whose boundary is  $P_{\beta}(J,\lambda)$. Coning the boundary to $J$ gives a rational slice surface $\Sigma_J$ where $\chi(\Sigma_K)=\chi(\Sigma_J)$.
\end{remark}

The rational Seifert genus is always achieved by some rational Seifert surface, and is therefore  a rational number.  This is a consequence of the linearity of the Thurston norm along rays through the origin, and the fact that there is a unique primitive slope on  $\partial \nu(K)$ which dies in the first of homology of the knot complement.  As we are forced to consider all braided patterns knots, and as there aren't analogous linearity properties for minimal genus problems in smooth 4-manifolds, the following question might be rather subtle. 

\begin{ques}
Is the rational slice genus always a rational number?
\end{ques}

We would also like to have a better understanding of knots whose rational slice genus is zero. In \cite[Theorem 2.11]{Calegari-Gordon}, Calegari and Gordon show the only knots whose rational Seifert surfaces can be a disk are null-homologous unknots in $Y$ or a knot that forms the core of a genus one Heegaard splitting of a lens space summand of $Y$. Connect summing one of these knots with a slice knot in the 3-sphere, yields many knots with slice surfaces that are disks, and whose rational slice genus therefore is zero. 

\begin{ques}
Are there knots whose minimal genus rational slice surfaces are disks and which are not contained in lens space summands?
\end{ques}

\subsection{Knot Floer homology and $\tau$ invariants}
Here we give a brief overview of knot Floer homology for rationally null-homologous knots. In the present context, we need only the ``hat version" of Floer homology. For further details see \cite{RelativeAdjunction}.  

In \cite{HolDisk}, Ozsv\'ath and Szab\'o define a complex $\CFa(Y)$ associated to a pointed Heegaard diagram $(\Sigma,\alphas,\betas,w)$ for a 3-manifold $Y$. This complex is generated over $\F=\Z\slash2\Z$ by intersection points $\x$ between Lagrangian tori $\Ta$ and $\Tb$ in the $g$-fold symmetric product of $\Sigma$ with itself, where $g$ is the genus of $\Sigma$. The differential is given by $$\hat\partial(\x)=\sum_{\y\in\Ta\cap\Tb}\sum_{\substack{\phi\in\pi_2(\x,\y)\\ \mu(\phi)=1 \\ n_w(\phi)=0}}\#\widehat{\mathcal{M}}(\phi) \cdot\y.$$
Here $\#\widehat{\mathcal{M}}(\phi)$ denotes the number of points modulo two in the unparameterized moduli space of pseudo-holomorphic disks from $\x$ to $\y$ in the homotopy class of $\phi$, $\mu(\phi)$ is the Maslov index, and $n_w(\phi)$ is the intersection number of $\phi$ with the complex codimension one subvariety $V_w$ in the symmetric product consisting of unordered tuples that contain the point $w$. The basepoint $w$ also determines a map $\spinc_w(-):\Ta\cap\Tb\to \SpinC(Y)$ along which the complex splits and we write $\CFa(Y,\spinc)$ for the summand corresponding to $\spinc$. Ozsv\'ath and Szab\'o show the chain homotopy type of $\CFa(Y,\spinc)$ is an invariant of the underlying 3-manifold.

In \cite{Knots} and \cite{RationalSurgeries} Ozsv\'ath and Szab\'o show an oriented rationally null-homologous knot $K$ in $Y$ gives rise to an addition filtration on $\CFa(Y)$. A choice of doubly pointed Heegaard diagram $(\Sigma,\alphas,\betas,w,z)$ contains the information of a relative Morse function for the pair $(Y,K)$, where the knot is realized geometrically as a union of flow lines $K=\gamma_w-\gamma_z$ from the unique index zero to the unique index three critical point of the Morse function. In the symmetric product, we count the intersections of $V_w$ and $V_z$ with the pseudo-holomorphic disks counted by the differential. More specifically, the additional $z$ basepoint determines a map $\spinc_{w,z}:\Ta\cap\Tb\to\SpinC(Y,K)$ through which the map $\spinc_w(-)$ factors.

The Alexander grading of a generator $\x$ in $\Ta\cap\Tb$ is $$A_{Y,K,[F]}(\x)=\frac{1}{2p}(\langle c_1(\spinc_{w,z}(\x)),[F]\rangle +[\mu]\cdot[F]),$$  and is extended to arbitrary chains in $\CFa(Y)$ by the convention that the Alexander grading of a linear combination is the maximum value of any of the constituent generators $\x\in \Ta \cap \Tb$.  In the definition, $[F]$ is the homology class of a rational Seifert surface for $K$.  Different choices of $[F]$ will shift the Alexander grading, as will different choices of trivialization of the $\SpinC$ structure along the boundary, see \cite[Remark 2.2]{RelativeAdjunction}.  Typically we suppress this dependence and write $A_{Y,K}$. 
 
 \begin{defn}
 Let $K$ be a rationally null-homologous knot in a 3-manifold $Y$ and $\alpha$ a nontrivial class in $\HFa(Y)$, then
 $$\tau_\alpha(Y,K):=\min\{A_{Y,K}(\x) \mid \x\in\CFa(Y) \text{ and } [\x]=\alpha\in\HFa(Y)\}.$$
 \end{defn}  
 \noindent We  simplify notation by writing $\tau_\alpha(K)$ when $Y$ is understood.

 In \cite{RelativeAdjunction} the authors also define $\tau$ invariants for links via Ozsv\'ath and Szab\'o's knotification construction. If $L\subset Y$ is a link, then the \emph{knotification} of $L$ is a knot $\kappa(L)$ in $Y\#^{|L|-1}\SoneStwo$ where $|L|$ is the number of components of the link, see \cite[Section 5.3]{RelativeAdjunction} or \cite[Section 2.1]{Knots} for more details. 
 \begin{defn}\label{defn:linktau}
     Let $L\subset Y$ be a rationally null-homologous link with $|L|$ components, $\alpha$ a nontrivial class in $\HFa(Y)$ and $\Theta$ a nontrivial class in $\HFa(\#^{|L|-1}\SoneStwo)$ then,
         $$\tau_{\alpha\otimes\Theta}(Y,L):=\tau_{\alpha\otimes\Theta}(Y\#^{|L|-1}\SoneStwo, \kappa(L)).$$
 \end{defn}

\subsection{The Alexander grading, rational linking, and $\tau$-invariants}

The Alexander grading is closely related to the $\Q\slash\Z$ linking form on $Y$. Relative $\SpinC$ structures are in affine correspondence with $H^2(Y,K)$, in direct analogy with the affine identification between  $\SpinC$ structures on  $Y$ and $H^2(Y)$.  Using this, one can  view a relative $\SpinC$ structure as a chosen lift of an absolute $\SpinC$ structure along the map $H^2(Y,K)\rightarrow H^2(Y)$.    In particular, a pair of relative $\SpinC$ structures $\xi_1$ and $\xi_2$ on $Y$ determine a well-defined cohomology class $\xi_2-\xi_1$ in $H^2(Y,K)$ that maps to a class $PD(\varphi)$ in $H^2(Y)$. It follows that  \begin{equation}\label{eqn:selflink}\lk_{\Q\slash\Z}(\varphi,[K])=\frac{1}{p}\langle \xi_2-\xi_1,[F]\rangle= A_{Y,K}(\xi_2)-A_{Y,K}(\xi_1) \mod\Z
\end{equation}

In other words, the Alexander grading difference between generators in relative $\SpinC$ structure $\xi_1$ and $\xi_2$ is an integral lift of the linking form applied to the homology class $[K]$ and the difference of their underlying absolute $\SpinC$ structures.
In particular:
\begin{prop} For any rationally null-homologous knot $K$, 
$$\tau_{\max}(K)-\tau_{\min}(K)\geq \underset{\varphi\in \Tor H_1(Y)} \max\{\lk_{\Q\slash\Z}(\varphi,[K])\}$$ 
where $\lk_{\Q\slash\Z}(\varphi,[K])$ is the representative of the $\Q\slash\Z$ linking form in $[0,1)$.  
\end{prop}

\begin{proof}

Let $\tau_{\min}(K)=\tau_{\alpha}(K)$ where $\alpha\in\HFa(Y,\spinc)$. Let $\gamma$ be a nontrivial class in $\HFa(Y,\spinc+PD(\varphi))$. Then $\tau_{\max}(K)-\tau_{\min}(K)\geq \tau_\gamma(K)-\tau_{\min}(K)\geq 0$ and $\tau_\gamma(K)-\tau_{\min}(K)=n+\lk_{\Q\slash\Z}(\varphi,[K])$. Thus $n\geq 0$ and we have $$\tau_{\max}(K)-\tau_{\min}(K)\geq\lk_{\Q\slash\Z}(\varphi,[K]). $$ \end{proof}
\begin{remark}
 The reader may be tempted to combine the proposition  with Theorem \ref{thm:RationalGenusBound} to obtain a lower bound for the rational slice  genus of $K$ in terms of the linking form: $$\underset{\varphi\in \Tor H_1(Y)} \max\{\lk_{\Q\slash\Z}(\varphi,[K])\}\leq 2\lVert K\rVert_{Y\times [0,1]}+1.$$ While this is true, the fact that the left hand side is never bigger than 1 renders the bound  useless.
\end{remark}

\section{$\tau$ invariants of $(p,pn+1)$-cables} In this section, we generalize previous work of the first author \cite{CableII} to bound the $\tau$ invariants of $(p,pn+1)$-cables of rationally null-homologous knots. Throughout, we fix a choice of longitude $\lambda$ for $K$ and consider the $(p,pn+1)$-cable of $K$ with respect to this choice, which we denote $K_{p,pn+1}$. We warn the reader that the specific integer $n$ in the cabling parameters depends on $\lambda$. However, we suppress this dependence to simplify notation.

\subsection{Heegaard diagrams for cables} Fix a doubly pointed Heegaard diagram $$H = (\Sigma,\alphas,\betas_0\cup\mu,w,z)$$ adapted to the knot $K$ in $Y$ so that the final $\beta$-curve is a meridian for $K$ and intersects the final $\alpha$-curve, $\alpha_g$ in a unique point: $\alpha_g\cap \mu=x_0$.

For any choice of positive integers $p$ and $n$, a Heegaard diagram for the $(p,pn+1)$-cable is given by $$H(p,n)=(\Sigma,\alphas,\betas_0\cup \tilde{\beta},w,z').$$ The only differences between $H$ and $H(p,n)$ are the final $\beta$-curve and second basepoint. In $H(p,n)$ the curve $\tilde{\beta}$ is obtained from $\mu$ by a finger move winding the meridian $p-1$ times along the $n$-framed longitude $\lambda_n=\lambda+n\mu$ for $K$. After this isotopy, the basepoint $z'$ is added so that the curve connecting $w$ to $z'$ has intersection number $p$ with $\tilde\beta$. See Figure \ref{fig:H(p,n)},  cf. \cite[Figure 1]{CableII}. 

\begin{figure}
\centering{
\includegraphics[width=.4\textwidth]{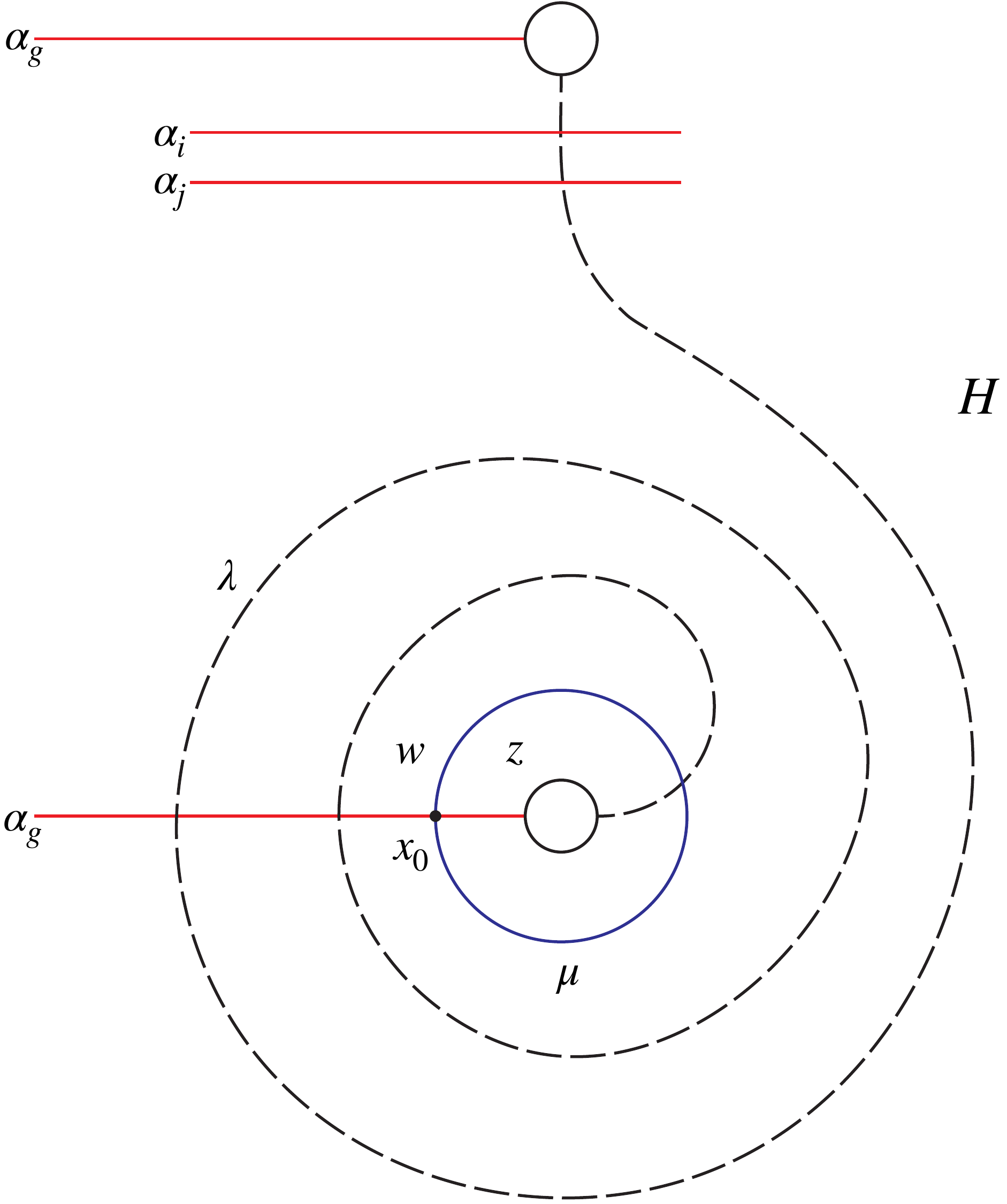}
\hfill
\includegraphics[width=.4\textwidth]{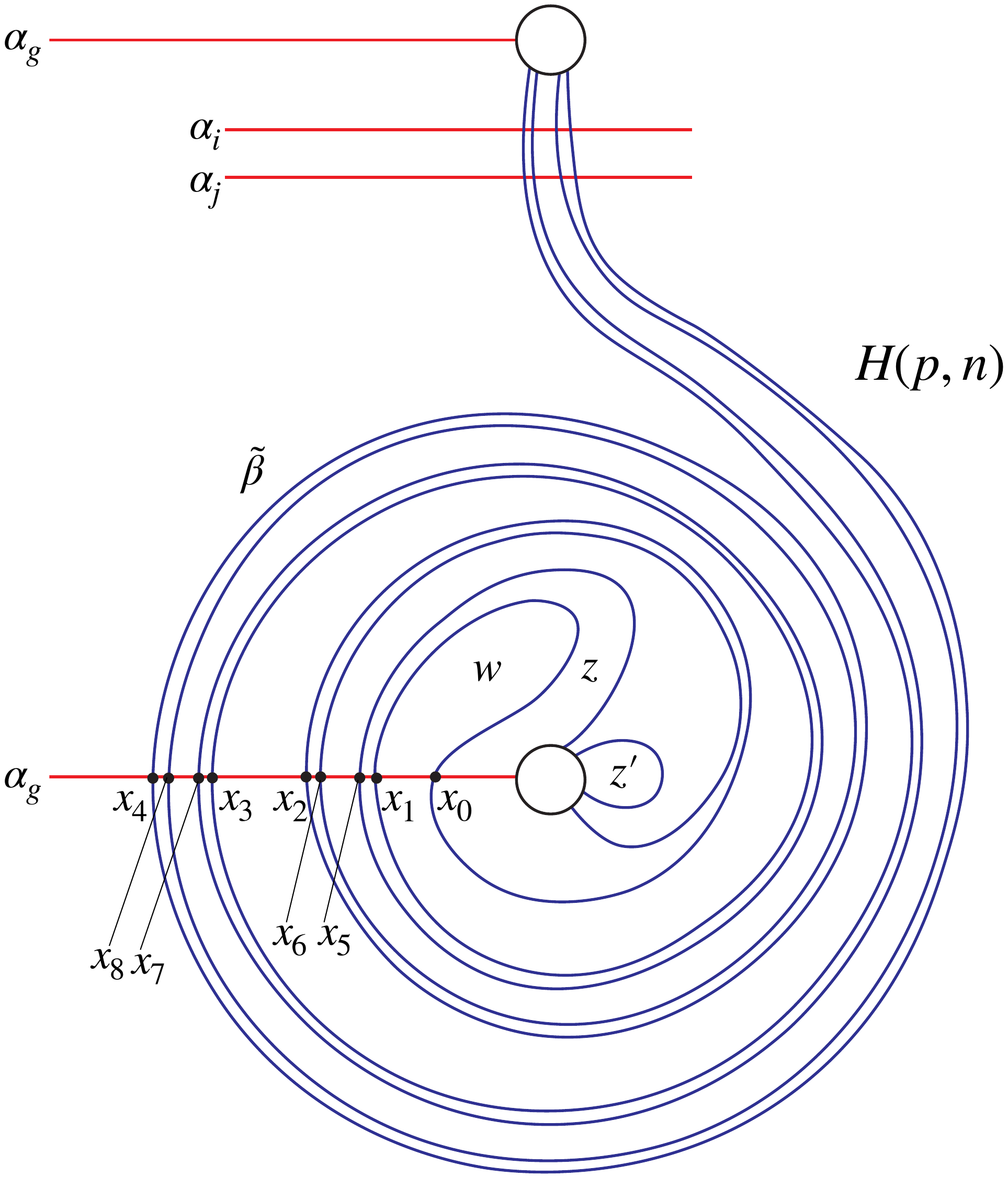}}
\caption{Model Heegaard diagrams for $H$ and $H(p,n)$.}
\label{fig:H(p,n)}
\end{figure}

As in \cite{CableII}, we consider $K$ and its $(p,pn+1)$ cable simultaneously by reintroducing the $z$ basepoint into the diagram $H(p,n)$.  With the additional basepoint,  $H(p,n)$ represents both $K$ and  $K_{p,pn+1}$, depending on whether we use the pair $(w,z)$ or $(w,z')$, respectively.  

Using $H(p,n)$ to calculate the Floer homology of both knots, we establish a correspondence between certain generators of their knot Floer chain complexes. Given a generator $(\y,x_0)$ for $H$, there are corresponding generators of the form $(\y, x_i)$ where for $i=0, \ldots 2n(p-1)$, $x_i$ is an intersection point between $\alpha_g$ and $\tilde\beta$. Keeping with the terminology of \cite{CableII} we refer to generators of the form $(\y,x_i)$ as \emph{exterior}. The ordering of the generators is determined by Figure \ref{fig:H(p,n)} and follows the conventions in \cite{CableII}. We call intersection points of the form $(\y, x_0)$ \emph{outermost intersection points} for $H(p,n)$. Clearly, the number of such generators in $H(p,n)$ corresponding to $(\y,x_0)$ in $H$ increases as $n$ grows.

Let 
\begin{align*}
C(i)&:= \{\y \mid (\y, x_0) \text{ has Alexander grading $i$ in $H$}\}\\
A&:= \text{ Alexander grading for generators of $H(p,n)$ with respect to $(w,z)$ }\\
A'&:= \text{ Alexander grading for generators of $H(p,n)$ with respect to $(w,z')$}
\end{align*}

We first compare Alexander gradings of the outermost intersection points. 

\begin{lemma} [Lemma 2.5 of \cite{CableII}]\label{lemma:AgradingOutermost}
If $\x=(\y, x_0)$ is an outermost intersection point such that $\y$ is in $C(i)$, then the following hold:
\begin{enumerate}
\item The $\SpinC$ structure $\spinc$ associated to $\x$ is independent of $z$ or $z'$ and agrees with the $\SpinC$ structure associated to the corresponding generator for $H$.
\item The Alexander grading of $(\y,x_0)$ with respect to $(w,z)$ is the same in $H$ as it is in $H(p,n)$. In other words: $$A(\y,x_0)=i.$$
\item $\displaystyle A'(\y,x_0)=pA(\y,x_0)+ \frac{p(p-1)\nlink}{2}$ where $\nlink$ is the rational linking number of $K$ and $\lambda_n$. 
\end{enumerate}
\end{lemma}

\begin{proof}Statements (1) and (2) follow directly from \cite{CableII}. To prove (3), we compute the Alexander gradings using Proposition 1.3 of \cite{Hedden-Levine}. 

Let $q$ be the order of $[K]$ in $H_1(Y)$. Fix a rational Seifert surface $F$ of order $q$ and a choice of framing $\lambda$ for $K$. By Proposition \ref{prop:RatSS}, there exists an integer $r$ such that $\partial F=q\lambda+r\mu$ and $\lk_{\Q}(K,\lambda)=-\frac{r}{q}$.

Let $\sP$ be a relative periodic domain representing the class of $F$ in the Heegaard diagram $(\Sigma, \alphas,\betas_0\cup\tilde\beta, w,z)$. A relative periodic domain $\sP$ is, by definition, a relative 2-chain given by $\sP=\sum_i a_iD_i$ where $D_i$ are connected components of $\Sigma-(\alphas\cup\betas_0\cup\tilde\beta\cup\lambda)$ appearing with multiplicity $a_i$ and 
$$\partial \sP=q\lambda+r\tilde\beta+q\alpha_g+\sum_{i=1}^{g-1}(n_i\alpha_i+m_i\beta_i)$$
where $n_i$ and $m_i$ are the multiplicities of $\alpha_i$ and $\beta_i$ respectively. 

Proposition 1.3 of \cite{Hedden-Levine} shows:
\begin{equation}\label{eqn:Agrading}
2qA(\y,x_0)= \hat\chi(\sP)+2n_{(\y,x_0)}(\sP)-n_z(\sP)-n_w(\sP).
\end{equation}
Write $\gamma$ for the multiplicity of $\sP$ in the region indicated in Figure \ref{fig:windingregion}. We calculate the multiplicities in nearby components relative to $\gamma$ within the winding region. This calculation shows:
\begin{align*}
n_{w}(\sP)+n_{z}(\sP)=2\gamma+q+r\\
2n_{x_0}(\sP)=2\gamma+q+r.
\end{align*}
Recalling that $2n_{(\y,x_0)}(\sP)=2n_{\y}(\sP)+2n_{x_0}(\sP)$ we conclude: $$2qA(\y,x_0)= \hat\chi(\sP)+2n_{\y}(\sP).$$

At the same time, consider a relative periodic domain $\tilde\sP$ for $K_{p,pn+1}$, which is $p\sP$. The longitude for $K_{p,pn+1}$ can be written in terms of $\lambda$ as $$\lambda_{p,n}=p\lambda+(pn+1)\tilde\beta.$$ Then $\tilde\sP$ represents the class of a rational Seifert surface $[F_{p,pn+1}]$ for $K_{p,pn+1}$ and:
\begin{align*}
\partial \tilde\sP&=q\lambda_{p,n}+(pr-q(pn+1))\tilde\beta+pq\alpha_g+\sum_{i=1}^{g-1}(pn_i\alpha_i+pm_i\beta_i).\\
&=q\lambda_{p,n}+N\tilde\beta+pq\alpha_g+\sum_{i=1}^{g-1}(pn_i\alpha_i+pm_i\beta_i).
\end{align*} 
where we shorten the expression by writing $N=pr-q(pn+1)$.

Proposition 1.3 of \cite{Hedden-Levine} applied to $\tilde{\sP}$ now gives:

\begin{equation}\label{eqn:A'grading}
2qA'(\y,x_0)=\hat\chi(\tilde{\sP})+2n_{(\y,x_0)}(\tilde{\sP})-n_{z'}(\tilde\sP)-n_{w}(\tilde\sP).
\end{equation}

\begin{figure}
\includegraphics[width=\textwidth]{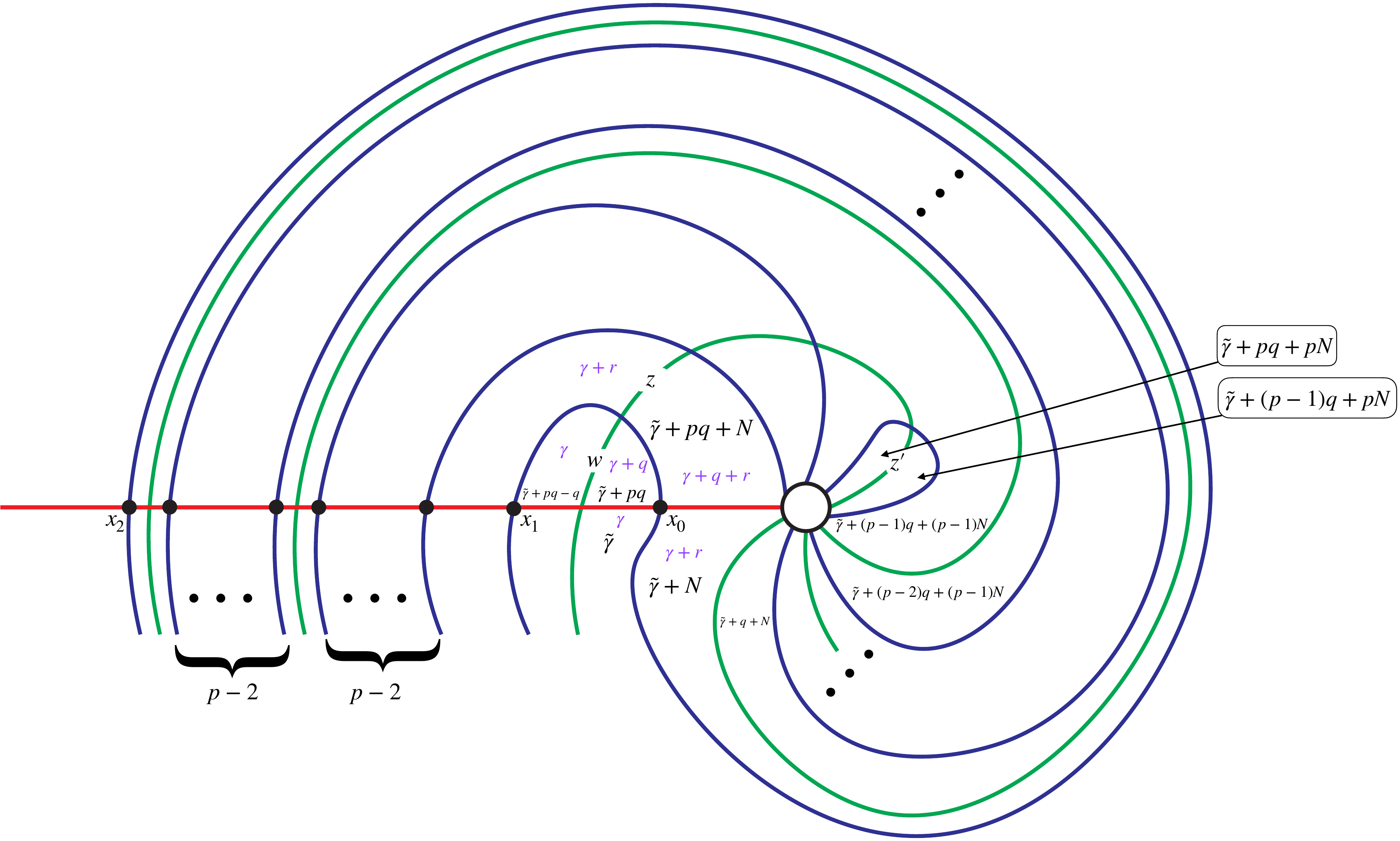}
\caption{The Heegaard diagram $H(p,n)$ near $x_0$. Multiplicities for $\tilde{\sP}$ are indicated in black. Multiplicities for $\sP$ are in purple. }
\label{fig:windingregion}
\end{figure}
 
We refer to the annular neighborhood on $H(p,n)$ of the  meridian of $K$ in which its $n$-framed longitude wraps $n$ times as the {\em winding region}.  Outside of the winding region the multiplicities of $\tilde\sP$ are $p$ times the multiplicities of $\sP$, $2n_{\y}(\tilde\sP)=2pn_{\y}(\sP)$. A model calculation of multiplicities inside the winding region shows $\hat\chi(\tilde \sP)=p\hat\chi(\sP)$.
 
Putting Equations \ref{eqn:Agrading} and \ref{eqn:A'grading} together now gives:
\begin{align*}
2qA'(\y,x_0)&=p\hat\chi(\sP)+2pn_{\y}(\sP)+2n_{x_0}(\tilde{\sP})-n_{ z'}(\tilde\sP)-n_{w}(\tilde\sP)\\
&=2pqA(\y,x_0)+2n_{x_0}(\tilde{\sP})-n_{z'}(\tilde\sP)-n_{w}(\tilde\sP).
\end{align*}

To complete the proof, we will show: $$2n_{x_0}(\tilde{\sP})-n_{z'}(\tilde\sP)-n_{w}(\tilde\sP)=p(p-1)q\nlink.$$

Figure \ref{fig:windingregion} indicates the multiplicities of $\tilde\sP$ near the point $x_0$. Write $\tilde\gamma$ for the multiplicity of the region indicated in Figure \ref{fig:windingregion} with respect to $\tilde\sP$. We compute the other multiplicities in the winding region relative to this one. We have: 
\begin{align*}
n_{z'}(\tilde\sP)+n_{w}(\tilde\sP)= 2\tilde\gamma+2pq+pN-q\\
2n_{x_0}(\tilde\sP)=2\tilde\gamma+pq+N.
\end{align*}
Thus, 
\begin{align*}
2n_{x_0}(\tilde\sP)-n_{z'}(\tilde\sP)-n_{w}(\tilde\sP)&=N-pq-pN+q\\
&=(1-p)(q+N)\\
&=p(qn-r)(p-1)\\
&=p(p-1)q\nlink.
\end{align*}
\end{proof}
\noindent The following lemmas from \cite{Cabling} immediately generalize. 

\begin{lemma} [Lemma 3.3 \cite{Cabling}]\label{lemma:AgradingExterior}
Let $H(p,n)$ be as above. If $i$ is odd and $i<2n$, then
\begin{align*}
A({\y, x_{i-1}})-A({\y}, x_i)&= A'(\y,x_{i-1})-A'(\y,x_i)=1,\\
A(\y, x_{i})-A(\y , x_{i+1})&=0,\\
A'(\y, x_{i})-A'(\y, x_{i+1})&= p-1.
\end{align*}

\end{lemma} 

\begin{lemma} [Lemma 3.4 \cite{Cabling}]\label{lemma:AgradingExterior2}
If ${\bf y}$ in $C(j)$ and ${\bf z}$ in $C(k)$, then
\begin{align*}
A(\y, x_{i})-A(\z, x_{i})&= j-k,\\
A'(\y, x_{i})-A'(\z, x_{i})&= p(j-k).
\end{align*}

\end{lemma}

When the cabling parameter $n$ is sufficiently large, the exterior generators become the generators whose Alexander grading is largest.  

\begin{lemma}[Compare Lemma 2.8 \cite{CableII}]\label{lemma:maxa}
Fix an integer $l$, and a $\SpinC$ structure $\spinc$ on $Y$. Write $$\maxa_\spinc=\max\{A(\x)\;  |\; \x \text{ is any intersection point with $\spinc_{w,z}(\x)=\spinc$ exterior or not }\}.$$
Then there exists a constant $N>0$ such that for all $n$ with $n>N$, the only intersection points with $A(\x)\geq \maxa_\spinc-l$ are exterior.
\end{lemma}

\begin{figure}
\renewcommand\arraystretch{1.5}
\begin{tabular}{c|c|c|c|c}
& $C(\maxa_\spinc)$ & $C(\maxa_\spinc-1)$ & $C(\maxa_\spinc-2)$ & $\ldots$ \\
\hline
$x_0$& $(\maxa_\spinc,\maxa'_\spinc)$ & $(\maxa_\spinc-1, \maxa'_\spinc-p)$ & \cellcolor{black!25} $(\maxa_\spinc-2,\maxa'_\spinc-2p)$&$\ldots$\\
\hline
$x_1$& $(\maxa_\spinc-1,\maxa_\spinc'-1)$&\cellcolor{black!25} $(\maxa_\spinc-2, \maxa_\spinc'-p-1)$& \cellcolor{black!5}$(\maxa_\spinc-3,\maxa_\spinc'-2p-1)$&$\ldots$\\
\hline
$x_2$& $(\maxa_\spinc-1,\maxa_\spinc'-p)$ & \cellcolor{black!25}$(\maxa_\spinc-2, \maxa_\spinc'-2p)$&\cellcolor{black!5}$(\maxa_\spinc-3,\maxa_\spinc'-3p)$ &$\ldots$\\
\hline
$x_3$& \cellcolor{black!25} $(\maxa_\spinc-2,\maxa_\spinc'-p-1)$ & \cellcolor{black!5} $(\maxa_\spinc-3, \maxa_\spinc'-2p-1)$&\cellcolor{black!5}$(\maxa_\spinc-3,\maxa_\spinc'-3p-1)$ &$\ldots$\\
\hline
$x_4$& \cellcolor{black!25}$(\maxa_\spinc-2,\maxa_\spinc'-2p)$ & \cellcolor{black!5}$(\maxa_\spinc-3, \maxa_\spinc'-3p)$& \cellcolor{black!5}$(\maxa_\spinc-3,\maxa_\spinc'-4p)$ &$\ldots$\\
$\vdots$&$\vdots$&$\vdots$&$\vdots$&$\vdots$\\
\end{tabular}
\caption{Here we have a comparison of the $A$ and $A'$ gradings of exterior points --- $\maxa_\spinc$ is the maximum $A$ grading achieved in a given $\SpinC$ structure and $\maxa'_{\spinc}$ represents the $A'$ grading of the same generator. By Lemma \ref{lemma:AgradingOutermost} we know that $\maxa'=p\maxa+\frac{p(p-1)\nlink}{2}$. Incorporating this information with Lemmas \ref{lemma:AgradingExterior} and \ref{lemma:AgradingExterior2} allows us to construct the table.}
\label{fig:Agradings}
\end{figure}

Assume $n$ is large enough to apply Lemma \ref{lemma:maxa}. Then the table in Figure \ref{fig:Agradings} gives the Alexander gradings of exterior generators with respect to the $A$-grading and the $A'$-grading. Analyzing the knot filtration with respect to both gradings we obtain the following theorem:

\begin{theorem}\label{thm:PosCables}
Let $K\subset Y$ be a rationally null-homologous knot and $\spinc$ a $\SpinC$ structure on $Y$. Pick any $M\in \left[\frac{1}{q}\right]\Z$. Then there exists a constant $N(M)>0$ so that for all $n>N(M)$ the following holds for each $j>M$:

\begin{equation*}
H_*\left(\mathcal{F}_{\spinc}\left(Y,K_{p,pn+1}, pj+\frac{p(p-1)\nlink}{2}-1\right)\right)\cong H_*(\mathcal{F}_{\spinc}(Y,K,j-1)).
\end{equation*}
Furthermore, 
\begin{align*}
H_*\left(\Filt_\spinc\left(Y,K_{p,pn+1}, pj+\frac{p(p-1)\nlink}{2}-i\right)\right)\\\cong H_*\left(\Filt_\spinc\left(Y,K_{p,pn+1}, pj+\frac{p(p-1)\nlink}{2}-i-1\right)\right)
\end{align*}
for all $i=2,\ldots, p-1$.

In particular, 
\begin{equation*}
\tau_{\alpha}(Y,K_{p,pn+1})=\begin{cases} p\tau_\alpha(Y,K)+\frac{p(p-1)\nlink}{2}+(p-1) \text{ or }\\
									p\tau_\alpha(Y,K)+\frac{p(p-1)\nlink}{2}.\\
									\end{cases}
\end{equation*}
\end{theorem}

\begin{proof}
As in Theorem 2.2 of \cite{CableII}, the proof follows from analyzing the table in Figure \ref{fig:Agradings}, which shows us the Alexander gradings and corresponding filtration levels of our exterior generators. 

Observe that the diagram $H(p,n)$ is a diagram for $Y$ regardless of whether we use the $z$ or $z'$ basepoint, so the underlying set of generators for the corresponding knot Floer chain complex is the same. Moreover, the differential counts holomorphic Whitney disks with $n_w(\phi)=0$, which is also independent of $z$ and $z'$. Taking $l>\maxa_\spinc-M$ in Lemma \ref{lemma:maxa} and changing variables by writing $j-1=\maxa_\spinc-k$, we have the following isomorphisms of chain complexes for each $k<l$:
$$\Filt_\spinc\left(K_{p,pn+1},p\maxa_\spinc+\frac{p(p-1)\nlink}{2}-p(k-1)-1\right)\cong\Filt_\spinc(K,\maxa_\spinc-k)$$
and 
\begin{align*}
\Filt_\spinc\left(K_{p,pn+1},p\maxa_\spinc+\frac{p(p-1)\nlink}{2}-p(k-1)-i\right)&\quad\quad\\
\quad\quad\cong\Filt_\spinc\left(K_{p,pn+1},p\maxa_\spinc+\frac{p(p-1)\nlink}{2}-p(k-1)-i-1\right)
\end{align*}
for all $i=2,\ldots p-1$. This proves the first part of the theorem.

To show the stated the relationship between the $\tau_{\alpha}$'s, we recall that $$H_*(\Filt_{\spinc}(K, j))=0$$ when $j$ is sufficiently small, since the knot Floer homology is supported in a finite range of Alexander gradings such that $A_{\max}-A_{\min}=2\lVert K\rVert_Y+1$. Taking $M$ to be sufficiently negative, the statement now follows from the definition of $\tau_{\alpha}$. \end{proof}

We have an analogous theorem for negative cables:
\begin{theorem}\label{thm:NegCables}
Let $K\subset Y$ be a rationally null-homologous knot and $\spinc\in\SpinC(Y)$. Pick any $M\in [\frac{1}{q}]\Z$. Then there exists a constant $N(M)>0$ so that for all $n>N(M)$ the following holds for each $j>M$:
\begin{equation*}
H_*\left(\mathcal{F}_{\spinc}\left(Y,K_{p,p(-n)+1}, pj+\frac{p(p-1)\lk_{\Q}(K,\lambda_{-n})}{2}-1\right)\right)\cong H_*(\mathcal{F}_{\spinc}(Y,K,j-1)).
\end{equation*}
Furthermore, 
\begin{align*}
H_*\left(\Filt_\spinc\left(Y,K_{p,p(-n)+1}, pj+\frac{p(p-1)\lk_{\Q}(K,\lambda_{-n})}{2}+i\right)\right)\\\cong H_*\left(\Filt_\spinc\left(Y,K_{p,p(-n)+1}, pj+\frac{p(p-1)\lk_{\Q}(K,\lambda_{-n})}{2}+i+1\right)\right)
\end{align*}
for all $i=2,\ldots p-1$.

In particular, 
\begin{equation*}
\tau_{\alpha}(Y,K_{p,p(-n)+1})=\begin{cases} p\tau_\alpha(Y,K)+\frac{p(p-1)\lk_{\Q}(K,\lambda_{-n})}{2}+(p-1) \text{ or }\\
									p\tau_\alpha(Y,K)+\frac{p(p-1)\lk_{\Q}(K,\lambda_{-n})}{2}.\\
									\end{cases}
\end{equation*}
\end{theorem}

Combining Theorems \ref{thm:PosCables} and \ref{thm:NegCables} with the following crossing change inequality for $\tau$, we bound the $\tau$ invariants of $K_{p,pn+1}$ in terms of the $\tau$ invariants of $K$ for every integer $n$. 

\begin{prop}\cite[Proposition 5.12]{RelativeAdjunction}\label{prop:CrossingChange}
If rationally null-homologous links $L_-, L_+\subset Y$ differ by changing a single negative crossing in $L_-$ to a positive crossing in $L_+$, then $$\tau_{\alpha\otimes \Theta}(L_-)\leq \tau_{\alpha\otimes \Theta}(L_+)\leq \tau_{\alpha\otimes \Theta}(L_-)+1.$$
\end{prop} 

\begin{theorem}\label{thm:CableBounds}
Let $K\subset Y$ be rationally null-homologous, and let $p$ be a positive integer. For all integers $n$, 
$$p\tau_\alpha(K)+\frac{p(p-1)}{2}\nlink\leq \tau_\alpha(K_{p,pn+1}) \leq p\tau_\alpha(K)+\frac{p(p-1)}{2}\nlink+(p-1).
$$
\end{theorem}

\begin{proof} The proof essentially follows the proof of Theorem 1.2 of \cite{CableII}. Theorems \ref{thm:PosCables} and \ref{thm:NegCables} imply the result for $|n|$ sufficiently large. To obtain the result for the intermediate values of $n$, we first observe that the knot $K_{p,pn+1}$ can be changed into $K_{p,p(n-1)+1}$ by a sequence of $\frac{p(p-1)}{2}$ positive to negative crossing changes. Applying the crossing change inequality in Proposition \ref{prop:CrossingChange}, we work down from a positive cable with $n>>0$ to obtain the lower bound for all $n$. To obtain the upper bound, we work the other way, starting from a negative cable with $n<<0$, applying the crossing change formula as we go.
\end{proof}

\section{$\tau$ invariants of braid pattern links} 
We begin this section by proving a genus bound for cobordisms between links. This bound essentially follows from the main result of \cite{RelativeAdjunction}, once we construct a cobordism between the knotifications of the links. In the second part, we combine this genus bound with Theorem \ref{thm:CableBounds} to prove Theorem \ref{thm:BPBounds}. 

\subsection{Genus bounds for link cobordisms} Given a cobordism between a pair of links, we construct a cobordism between their knotifications. We use a combination of two different constructions, both of which are employed simultaneously to the surface and the ambient 4-manifold in which it is embedded. The first is a \emph{self arc sum}. The second  utilizes handle additions attached to the incoming and outgoing ends of the  cobordism.

To begin, a path cobordism between pointed 3-manifolds $(Y_0,p_0)$ and $(Y_1,p_1)$ is a pair $(W,\Gamma)$ where $W$ is a cobordism and $\Gamma$ is an embedded path in $W$ from $p_0$ to $p_1$. Given a second path cobordism $(W',\Gamma')$ between pointed 3-manifolds $(Y_0',p_0')$ and $(Y_1',p_1')$, we define the \emph{arc sum} to be the 4-manifold: $$ W\otimes W' := W\setminus\nu(\Gamma)\bigcup_{S^2\times [0,1]} W'\setminus\nu(\Gamma')$$ obtained by removing tubular neighborhoods of each arc and gluing $W$ and $W'$ along the $S^2\times [0,1]$ boundaries of the arcs. The arc sum is a cobordism from $Y_0\#Y_0'$ to $Y_1\#Y_1'$.  In  \cite[Theorem 3.2]{RelativeAdjunction} the authors showed that cobordism-induced maps on Floer homology satisfy a product formula under arc sums:
\begin{equation}\label{eq:kunnetharc}
    F_{W\otimes W'} = F_{W}\otimes F_{W'}.
\end{equation}

Given additional basepoints, a path cobordism from one \emph{doubly} pointed 3-manifold $(Y_0,p_0,p'_0)$ to another $(Y_1,p_1,p'_1)$ is a pair $(W,\Gamma)$ consisting of a cobordism $W$ and now a \emph{pair} of disjoint embedded arcs $\Gamma=\gamma\cup\gamma'$ with endpoints $\gamma(i)=p_i$ and $\gamma'(i)=p'_i$ for $i=0,1$. We define the \emph{self arc sum} to be the manifold $W_{\Gamma}$ formed by removing tubular neighborhoods $\nu(\gamma)$ and $\nu(\gamma')$ of both arcs and identifying the resulting $S^2\times [0,1]$ boundaries via an orientation reversing diffeomorphism. Equivalently, the self arc sum is the 4-manifold obtained by removing tubular neighborhoods of both arcs and attaching a copy of $S^2\times [0,1]\times [0,1]$ so that $S^2\times \{0\}\times [0,1]$ is identified with the boundary of $\nu(\gamma)$ and $S^2\times \{1\}\times [0,1]$ with the boundary of $\nu(\gamma')$: $$W_\Gamma := W\setminus\nu(\gamma\cup \gamma') \bigcup_{S^2\times \{0\}\times [0,1]\sqcup S^2\times \{1\}\times [0,1]} S^2\times [0,1]\times [0,1]$$

As a 4-manifold the self arc sum $W_\Gamma$ is a cobordism from the self connected sum of $Y_0$ to the self connected sum of $Y_1$, where the self connected sum of a manifold is constructed by removing two balls in a path component and identifying the resulting  boundary spheres with an orientation-reversing diffeomorphism. Note that the self connected sum of a 3-manifold $Y$ is diffeomorphic to the $Y\#\SoneStwo$. Similarly, for the self arc sum we have:
\begin{lemma}\label{lem:SelfArcSum} Let $(W,\Gamma)$ be a path cobordism with a pair of paths embedded in the same  component of $W$ and whose initial (resp. terminal) points are in the same  component of the incoming (resp. outgoing) boundary. Then the self arc sum $W_{\Gamma}$ is diffeomorphic to the (single) arc sum $$W \otimes (S^1\times S^2\times [0,1]).$$
\end{lemma}
\begin{proof}
To begin, we construct a  Morse function $f:W\to [0,1]$ adapted to the paths $\gamma$ and $\gamma'$ in a particular way. Let $\gamma:[0,1]\to W$ and $\gamma':[0,1]\to W$ denote the embeddings of each arc in $W$. Define $f$ first on a neighborhood of the arcs so that the compositions $f\circ\gamma$ and $f\circ\gamma'$ equal the identity map from $[0,1]$ to itself. Then, extend $f$ to a Morse function on all of $W$, and pick a Riemannian metric on $W$ with respect to which we consider the  gradient vector field associated to $f$. 

Now choose any path $\delta_0$ in $Y_0$ between the points $p_0$ and $p'_0$ which lies in the complement of the (non-separating) subset of $Y_0$ which limits to a critical point of $f$ under the flow of its gradient vector field. With such a choice, $\delta_0$ will flow in one unit of time to a path $\delta_1$ in $Y_1$. The restriction  of the flow to $\delta_0$  yields a smooth family of embedded paths $\delta_t$, $t\in [0,1]$ in $W$ beginning at $\delta_0$ and ending at $\delta_1$ 

The path $\delta_0$ has a tubular neighborhood in $Y_0$ that we can identify with $$[0,1]\times D^2\cong B^3\subset Y_0.$$ This identification, under the flow, induces an identification of a neighborhood of each arc $\delta_t$ with a smoothly embedded $B^3$ in the level set $Y_t=f^{-1}(t)$ (note that while critical levels will not be smooth 3-manifolds, $\delta_t$ and its neighborhood in $Y_t$ will lie in an open subset of $Y_t$ which is smooth). Taken together, these neighborhoods yield a smooth and proper embedding of $B^3\times [0,1]$ in $W$. 

Focusing on the embedded copy of $B^3\times [0,1]$ in $W$, to form $W_\Gamma$ one should remove $\nu(\gamma)\cup\nu(\gamma')$ from $B^3\times [0,1]$ and attach a copy of $S^2\times [0,1]\times [0,1]$. Note that removing two small balls from $B^3$ and attaching $S^2\times [0,1]$ yields a 3-manifold diffeomorphic to $S^1\times S^2\setminus B^3$. Doing this at each level set of $B^3\times [0,1]$ we see that, $$B^3\times [0,1]\setminus(\nu(\gamma)\cup \nu(\gamma'))\bigcup S^2\times [0,1]\times [0,1]$$ is diffeomorphic to $\SoneStwo\times [0,1]\setminus (B^3\times [0,1])$.
\end{proof}

For the second construction, involving handle additions, we first recall some notation from \cite[Section 5.2]{RelativeAdjunction}. A class $\Theta$ in $\HFa(\#^\ell \SoneStwo)$ can be written as a sum of classes of the form $\Theta=\theta_{\epsilon_1}\otimes\ldots\otimes\theta_{\epsilon_\ell}$ where $\epsilon_k\in\{+,-\}$. Let $\Theta_{top}$ and $\Theta_{bot}$ denote the classes where $\epsilon_k=+$, respectively $-$, for all $k$.

Starting with a 3-manifold $Y$, construct a cobordism to $Y\#^\ell \SoneStwo$ by attaching $\ell$ 4-dimensional 1-handles to $Y\times [0,1]$. The resulting 4-manifold is diffeomorphic to $Y\times[0,1]\bignatural^\ell S^1\times B^3$ and we denote it by $W_{top}$.

Similarly, starting with $Y\#^\ell\SoneStwo$ construct a cobordism to $Y$ by attaching $\ell$ 2-handles to $Y\#^\ell\SoneStwo\times [0,1]$ along circles $S^1\times \{p\}$ in the $l$ different $S^1\times S^2$ summands (note that such circles have a unique framing, up to isotopy). The resulting 4-manifold is diffeomorphic to $Y\times [0,1]\bignatural^\ell S^2\times D^2$ and we denote it by $\overline{W}_{top}$.
 
\begin{lemma} \label{lem:Endmaps}
For each nontrivial class $\alpha\in\HFa(Y,\spinc)$, there is a unique $\SpinC$ structure $\spinct\in \SpinC(W_{top})$ for which the map: $$F_{W_{top},\spinct}:\HFa(Y,\spinc)\to \HFa(Y\#^\ell\SoneStwo,\spinc\#\spinc_0)$$ is nontrivial on $\alpha$.  The map carries $\alpha$ to $\alpha\otimes \Theta_{top}$. Similarly, there is a unique $\SpinC$ structure $\bar{\spinct}\in \SpinC(\overline{W}_{top})$ for which the map:: $$F_{\overline{W}_{top},\bar{\spinct}}:\HFa(Y\#^\ell\SoneStwo,\spinc\#\spinc_0)\to \HFa(Y,\spinc)$$ has $\alpha$ in its image. The map carries $\alpha\otimes \Theta_{top}$ to $\alpha$. 
\end{lemma}

\begin{proof}
The statement about $W_{top}$ follows from the definition of the maps associated to 4-dimensional 1-handle attachments \cite[Section 4.3]{HolDiskFour}. The statement about $\overline{W}_{top}$ can be deduced by from the composition law for cobordism maps, together with the fact that $\overline{W}_{top}\circ W_{top}$ is diffeomorphic to the product $Y\times [0,1]$
\end{proof}


Now we state our genus bound for link cobordisms. Let $W$ be a cobordism between 3-manifolds $Y_0$ and $Y_1$ and $F_{W,\spinct}:\HFa(Y_0,\spinct\lvert_{Y_0})\to \HFa(Y_1,\spinct\lvert_{Y_1})$ the induced map on Floer homology. 

\begin{prop}\label{prop:betterlinkbound}  
 Let  $\Sigma\subset W$ be a smooth cobordism between links $L_0\subset Y_0$ and $L_1\subset Y_1$ for which both $\pi_0(L_i)\to \pi_0(\Sigma)$, $i=0,1$, are surjective. If $F_{W,\spinct}(\alpha)\neq 0$ 
then, 
\begin{equation*}
\langle c_1(\spinct),[\Sigma_{S_0,S_1}]\rangle+[\Sigma_{S_0,S_1}]^2+2\tau_{F_{W,\spinct}(\alpha)\otimes\Theta_{top}}(L_1)-2\tau_{\alpha\otimes\Theta_{top}}(L_0)\leq ||L_0|-|L_1||-\chi(\Sigma).
\end{equation*} 
\end{prop}
\begin{proof} Let $\ell=\min\{|L_0|,|L_1|\}-1$. We begin by fixing a collection of $\ell$ pairs of arcs, embedded on $\Sigma$. Each arc should have incoming end on $L_0$ and outgoing end on $L_1$ and, for each pair, the incoming and outgoing ends should lie on distinct link components. The surjectivity condition ensures that we can find such a collection of arcs.  

Now, we form the self-arc sum of $W$ along each pair of arcs. By Lemma \ref{lem:SelfArcSum}, the resulting manifold is diffeomorphic to $W\otimes (\#^l S^1\times S^2)\times [0,1]$. In addition, we form the self arc sum of $\Sigma$; that is, we surger $\Sigma$ along each pair of arcs, removing a small neighborhood of the arcs and replacing it with a pair of bands that pass through one of the $S^1\times S^2\times [0,1]$ factors.
The resulting surface has a knotified link on at least one side. Moreover, its Euler characteristic is unchanged under the surgery, and hence by an abuse of notation we continue to denote it by $\Sigma$. 

If $L_0$ and $L_1$ have a different number of components, we extend the 4-manifold further. Specifically, attach $|L_0|-|L_1|$ 4-dimensional 1-handles to the outgoing end if $|L_1|<|L_0|$. If instead, $|L_0|<|L_1|$ then attach $|L_1|-|L_0|$ 2-handles to the incoming end. The resulting 4-manifold $\overline{W}$ is a cobordism from $Y_0\#^{|L_0|-1}\SoneStwo$ to $Y_1\#^{|L_1|-1}\SoneStwo$ and the surface can be extended to a surface $\overline{\Sigma}$ from $\kappa(L_0)$ to $\kappa(L_1)$. 


We now consider the map on Floer homology induced by $\overline{W}$.  Lemma \ref{lem:SelfArcSum}  identifies the self-arc sum of $W$ with  $W\otimes (\#^lS^1\times S^2)\times [0,1]$ which, together with the product theorem for arc sums, Equation \eqref{eq:kunnetharc}, allows us to identify  the map on Floer homology induced by the self-arc sum  with
$$  {F_{W}\otimes Id}: \HFa(Y_0)\otimes\HFa(\#^\ell S^1\otimes S^2)\rightarrow \HFa(Y_1)\otimes\HFa(\#^\ell S^1\times S^2) $$

\noindent Combining this with Lemma \ref{lem:Endmaps}, we find that $F_{\overline{W},\bar{\spinct}}$ maps $\alpha\otimes \Theta_{top}$ to $F_{W,\spinct}(\alpha)\otimes \Theta_{top}$. We apply \cite[Theorem 1]{RelativeAdjunction} to $\overline\Sigma$ in $\overline{W}$ to obtain:  
\begin{align*}
\langle c_1(\bar\spinct),[\overline{\Sigma}_{S_{\kappa_0},S_{\kappa_1}}]\rangle+[\overline{\Sigma}_{S_{\kappa_0},S_{\kappa_1}}]^2+2\tau_{F_{W,\spinct}(\alpha)\otimes\Theta_{top}}(\kappa(L_1))-2\tau_{\alpha\otimes\Theta_{top}}(\kappa(L_0))&\leq 2g(\overline\Sigma)
\end{align*}
Where $\overline{\Sigma}_{S_{\kappa_0},S_{\kappa_1}}$ is a rational 2-chain formed from $\overline\Sigma$ and rational Seifert surfaces for $\kappa(L_0)$ and $\kappa(L_1)$ respectively.

Finally, we observe: $$\langle c_1(\bar\spinct),[\overline{\Sigma}_{S_{\kappa_0},S_{\kappa_1}}]\rangle+[\overline{\Sigma}_{S_{\kappa_0},S_{\kappa_1}}]^2=\langle c_1(\spinct),[\Sigma_{S_0,S_1}]\rangle +[\Sigma_{S_0,S_1}]^2.$$\end{proof}
Using a monotonicity property of $\tau$ invariants for links we can quickly deduce additional bounds for $\Theta_0$ and $\Theta_1$ which are not necessarily top Floer classes. We emphasize that the bound in Proposition \ref{prop:betterlinkbound} is better. However, there may be instances where comparing $\tau$'s for different $\Theta$'s is useful.

\begin{cor}\label{cor:LinkCobordisms}
Let  $\Sigma\subset W$ be a smooth cobordism between links $L_0\subset Y_0$ and $L_1\subset Y_1$ for which both $\pi_0(L_i)\to \pi_0(\Sigma)$ are surjective. If $F_{W,\spinct}(\alpha)\neq 0$, then: 
\begin{equation*}
\langle c_1(\spinct),[\Sigma_{S_0,S_1}]\rangle+[\Sigma_{S_0,S_1}]^2+2\tau_{F_{W,\spinct}(\alpha)\otimes\Theta_1}(L_1)-2\tau_{\alpha\otimes\Theta_0}(L_0)\leq -\chi(\Sigma)+|L_0|+|L_1|-2
\end{equation*} 
for any choice of $\Theta_0$ in $\HFa(\#^{|L_0|-1}S^1\times S^2)$ and $\Theta_1$ in $\HFa(\#^{|L_1|-1}S^1\times S^2)$.  
\end{cor}

\begin{proof}
 The $\tau$ invariants of links satisfy a monotonicity property \cite[Theorem 2(b)]{RelativeAdjunction}, which states  that for any class $\alpha$ in $\HFa(Y)$ and $\Theta$ in $
\HFa(\#^{|L|-1}\SoneStwo)$: $$2\tau_{\alpha\otimes \Theta_{bot}}(L)\leq 2\tau_{\alpha\otimes \Theta}(L)\leq 2\tau_{\alpha\otimes \Theta_{top}}(L)\leq 2\tau_{\alpha\otimes \Theta_{bot}}(L)+|L|-1.$$

\noindent We combine the bounds above with those of Proposition \ref{prop:betterlinkbound}. Monotonicity allows us to replace $\tau_{F_{W,\spinct}(\alpha)\otimes \Theta_{top}}(L_1)$ with $\tau_{F_{W,\spinct}(\alpha)\otimes\Theta_1}(L_1)$ in Proposition \ref{prop:betterlinkbound}, retaining its validity. Turning  the monotonicity bounds around, we have $$-\tau_{\alpha\otimes \Theta_0}(L_0)-|L_0|+1\leq -\tau_{\alpha\otimes \Theta_{bot}}(L_0)-|L_0|+1\leq -\tau_{\alpha\otimes \Theta_{top}}(L_0),$$
which we can also substitute in Proposition \ref{prop:betterlinkbound}, obtaining 
\begin{align*}
\langle c_1(\spinct),[\Sigma_{S_0,S_1}]\rangle+[\Sigma_{S_0,S_1}]^2+2\tau_{F_{W,\spinct}(\alpha)\otimes\Theta_1}(L_1)-2\tau_{\alpha\otimes\Theta_0}(L_0)&\leq ||L_0|-|L_1||-\chi(\Sigma)\\
& \hspace{2.5cm} +|L_0|-1\\
&\leq -\chi(\Sigma)+|L_0|+|L_1|-2.
\end{align*} 
\end{proof}

\subsection{Bounds for $\tau$ invariants of braid pattern links} We now prove Theorem \ref{thm:BPBounds}, which we restate here for the reader's convenience.

\BPBounds*

\begin{proof} To start, we fix some notation. Write $k$ for the number of positive crossings in the braid $\beta$ and $\ell$ for the number of negative crossings. Then the writhe is $\writhe(\beta)=k-\ell$ and the length of $\beta$ is $\len(\beta)=k+\ell$. Write $\beta_+$ for the positive braid obtained by changing all negative crossings of $\beta$ to positive. Write $\beta_-$ for the negative braid obtained by changing all positive crossings to negative. 

Fixing $N\gg 0$ we can construct a cobordism in $S^1\times D^2\times [0,1]$ from $\beta_+$ to the torus knot $T_{p,pN+1}$, which is given by attaching $(p-1)(pN+1)-\len(\beta)$ bands to the closure of $\beta_+$. We use the same collection of band attachments to produce a cobordism $\Sigma_+$ from $P_{\beta_+}(K,\lambda)$ to $K_{p,pN+1}$, contained in $\nu(K)\times [0,1]\subset Y\times [0,1]$. Note that $\Sigma_+$ is connected.

Applying Corollary \ref{cor:LinkCobordisms} to $\Sigma_+$ cobordism gives: 
\begin{align*}
    2\tau_\alpha(K_{p,pN+1})- 2\tau_{\alpha\otimes\Theta}(P_{\beta_+}(K,\lambda)) &\leq -\chi(\Sigma_+)+|\beta|-1\\
    &= (p-1)(pN+1)-\len(\beta)+\lvert\beta\rvert -1
\end{align*} where $\lvert\beta\rvert$ denotes the number of components of the closure of $\beta$.

Now, Theorem \ref{thm:CableBounds} shows that $2\tau_\alpha(K_{p,pN+1})$ is bounded below by $$2p\tau_{\alpha}(K)+p(p-1)\Nlink=2p\tau_{\alpha}(K)+(p-1)(p\lk_\Q(K,\lambda)+pN).$$ Thus,  
\begin{equation}\label{Eqn:posbound1}
-2\tau_{\alpha\otimes\Theta}(P_{\beta_+}(K,\lambda))+2p\tau_{\alpha}(K)+(p-1)p\lk_\Q(K,\lambda)+\len(\beta) \leq (p-1)+\lvert\beta\rvert-1.
\end{equation}

To obtain the bound for $P_{\beta}(K,\lambda)$, observe that $P_{\beta}(K,\lambda)$ can be changed to $P_{\beta_+}(K,\lambda)$ by changing $\ell$ negative crossings to positive. Applying Proposition \ref{prop:CrossingChange} a total of $\ell$ times gives the bound: $$\tau_{\alpha\otimes\Theta}(P_{\beta_+}(K,\lambda))\leq \tau_{\alpha\otimes\Theta}(P_{\beta}(K,\lambda))+\ell.$$

Combining with Equation \ref{Eqn:posbound1} we have:
$$-2\tau_{\alpha\otimes\Theta}(P_{\beta}(K,\lambda))-2\ell+2p\tau_{\alpha}(K)+(p-1)p\lk_\Q(K,\lambda)+k+\ell\leq (p-1)+\lvert\beta\rvert-1.$$
And thus, 
\begin{equation}\label{Eqn:posbound}
-2\tau_{\alpha\otimes\Theta}(P_{\beta}(K,\lambda))+2p\tau_{\alpha}(K)+(p-1)p\lk_\Q(K,\lambda)+\writhe(\beta)\leq (p-1)+\lvert\beta\rvert-1.
\end{equation}

Now, to obtain the other inequality, we make a similar argument for $\beta_-$. Start by fixing $N'\ll 0$. Then we can find a cobordism from $T_{p,pN'+1}$ to $\beta_-$ inside $S^1\times D^2\times [0,1]$. This time we attach bands to \emph{delete} negative crossings. The minimal total number of bands needed is $\len(T_{p,pN'+1})-\len(\beta)=-(p-1)(pN'+1)-\len(\beta)$. As before, we use the same sequence of band additions to construct a cobordism $\Sigma_-$ from $K_{p,pN'+1}$ to $P_{\beta_-}(K,\lambda)$ inside $\nu(K)\times [0,1]\subset Y\times [0,1]$. 

Corollary \ref{cor:LinkCobordisms} applied to $\Sigma_-$ implies: 
\begin{align*}
2\tau_{\alpha\otimes\Theta}(P_{\beta_-}(K,\lambda))- 2\tau_\alpha(K_{p,pN'+1})&\leq-\chi(\Sigma_-)+|\beta|-1\\
&=-(p-1)(pN'+1)-\len(\beta)+\lvert\beta\rvert-1.
\end{align*}
Theorem \ref{thm:CableBounds} implies $-2\tau_\alpha(K_{p,pN'+1})$ is bounded below by $$-2p\tau_{\alpha}(K)-(p-1)p\Nlink-2(p-1).$$
Thus:
\begin{equation}\label{Eqn:negbound1}
2\tau_{\alpha\otimes\Theta}(P_{\beta_-}(K,\lambda))-2p\tau_\alpha(K)-(p-1)p\selflink(K,\lambda)+\len(\beta)\leq (p-1)+\lvert\beta\rvert-1.
\end{equation}
Now $K_{\beta_-}$ can be changed to $P_{\beta}(K,\lambda)$ by changing $k$ negative crossings to positive. Applying Proposition \ref{prop:CrossingChange} a total of $k$ times gives:
$$\tau_{\alpha\otimes\Theta}(P_{\beta}(K,\lambda))-k\leq \tau_{\alpha\otimes\Theta}(P_{\beta_-}(K,\lambda)).$$
Finally, combining with Equation \ref{Eqn:negbound1}:
\begin{equation}\label{Eqn:negbound}
2\tau_{\alpha\otimes\Theta}(P_{\beta}(K,\lambda))-2p\tau_\alpha(K)-(p-1)p\selflink(K,\lambda)-\writhe(\beta)\leq (p-1)+\lvert\beta\rvert-1
\end{equation}
Together, Equations \ref{Eqn:posbound} and \ref{Eqn:negbound} imply the result.
\end{proof}

\section{A 4-dimensional rational genus bound}
In this section we combine our preparatory work from previous sections to prove our main result, Theorem \ref{thm:RationalGenusBound}.  We then observe that we  obtain stronger genus bounds if we restrict the boundary conditions allowed for rational slice surfaces.  In particular, we obtain Corollary \ref{cor:RatLongframed}, which gives bounds for the rational slice genus relative to the rational longitude. Finally, observing that $\tau_{\max}-\tau_{\min}$ is unchanged under local knotting, we show that this difference is actually a rational PL slice genus bound, proving Theorem \ref{cor:rationalPLslice}.

We begin with a lemma:
\begin{lemma}\label{lem:twisting}
Let $\lambda$ be a choice of framing for a knot $K\subset Y$ and $\beta$ a braid. If $P_{\beta}(K,\lambda)$ denotes the satellite formed from by identifying the longitude of a solid torus containing the closure of $\beta$ with the framing $\lambda$ for $K$, then $P_\beta(K,\lambda)$ and $P_{\beta(\Delta^2)^{-1}}(K,\lambda+\mu)$ describe the same satellite link of $K$, where $(\Delta^2)^{-1}$ denotes a negative full twist and $\mu$ is the meridian of $K$.
\end{lemma}
\begin{proof}
Identifying the solid torus containing the closure of $\beta$ with $\lambda+\mu$ puts a positive full twist in the resulting satellite knot. To undo this twist and obtain the original link $P_{\beta}(K,\lambda)$, change the pattern braid by a negative full twist.
\end{proof}

Let $$\tau_{\max}(K)=\max\{\tau_\alpha(K)\mid \alpha\in\HFa(Y)\}$$ and define $\tau_{\min}(K)$ analogously.  Recall the statement of Theorem \ref{thm:RationalGenusBound}.

\RationalGenusBound*

\begin{proof} Let $\Sigma$ be a rational slice surface for $K$. By Proposition \ref{prop:neighborhood}, near $K$ the surface induces a well-defined braided satellite of $K$. Choosing a framing $\lambda$ for $K$, there is a closed braid $\beta$ so that this satellite is written as $P_\beta(K,\lambda)$. Let \[c:=(p-1)p\selflink(K,\lambda)+\omega(\beta ).\]
\noindent The value of $c$ is a constant depending only on the satellite link itself, not the choice of $\lambda$ or $\beta$ used to describe this link. Indeed, Lemma \ref{lem:twisting} shows that $P_\beta(K,\lambda)$ has an equivalent description: $P_{\beta(\Delta^2)^{-1}}(K,\lambda+\mu)$. These changes impact the rational linking number and the writhe accordingly: $\selflink(K,\lambda+\mu)=\selflink(K,\lambda)+1$ and $\omega(\beta(\Delta^2)^{-1})=\omega(\beta)-(p-1)p$. It follows that the value of $c$ is independent of the choices used in the description of $P_\beta(K,\lambda)$.  

For the rest of the proof, we work with a fixed description of the induced satellite $P_{\beta}(K,\lambda)$, and we suppress the framing from the notation for simplicity: $P_\beta(K)$. In addition, we write $|\beta|$ for the number of components of $\beta$, which equals the number of components of $P_\beta(K)$.

Now, construct a pair of closed braids $\beta_{1}$ and $\beta_2$ from $\beta$ by introducing $\lvert\beta\rvert-1$ positive, respectively negative, crossings so that $\beta_1$  and $\beta_2$ are both \emph{knots} and 
\begin{align*}
\omega(\beta_1)&=\omega(\beta)+\lvert\beta\rvert-1\\
\omega(\beta_2)&=\omega(\beta)-\lvert\beta\rvert+1.
\end{align*} 
Then construct a pair of surfaces $\Sigma_1$ and $\Sigma_2$ from $\Sigma$ by attaching $\lvert\beta\rvert-1$ bands so that $\Sigma_1$ induces the satellite $P_{\beta_1}(K)$, $\Sigma_2$ induces $P_{\beta_2}(K)$ and for $i=1,2$: 
\begin{equation*}
-\chi(\Sigma_i)=-\chi(\Sigma)+\lvert\beta\rvert-1.
\end{equation*}
Applying \cite[Corollary 5.2]{RelativeAdjunction} to $\Sigma_i$ gives:
\begin{equation*}
2|\tau_{\alpha}(P_{\beta_i}(K))|\leq-\chi(\Sigma_i)+1= -\chi(\Sigma)+\lvert\beta\rvert. 
\end{equation*}
for each nontrivial class $\alpha$ in $\HFa(Y)$. Thus, over all nontrivial $\alpha$ and $\gamma$ in $\HFa(Y)$: 
\begin{equation}\label{eqn:Betai}
2\max_{\alpha,\gamma}\{|\tau_{\alpha}(P_{\beta_1}(K))|,|\tau_{\gamma}(P_{\beta_2}(K))|\}\leq -\chi(\Sigma)+\lvert\beta\rvert. 
\end{equation}

\begin{lemma}\label{lem:Boundmaxmin} With the set up above, we have the following estimate:
\begin{align*}\max_{\alpha,\gamma}\left\{ \begin{array}{r}
2p\tau_{\alpha}(K)+c-(p-1)+|\beta|-1\\
 -2p\tau_{\gamma}(K)-c-(p-1)+|\beta|-1
 \end{array}\right\}
\leq 2\max_{\alpha,\gamma}\{|\tau_{\alpha}(P_{\beta_1}(K)|,|\tau_{\gamma}(P_{\beta_2}(K)|\}.
\end{align*}
\end{lemma} 
\begin{proof}
The bound in Theorem \ref{thm:BPBounds} applied to $P_{\beta_1}(K)$ yields:  
\begin{align*}
\max_{\alpha}\left\{ \begin{array}{r}
2p\tau_{\alpha}(K)+p(p-1)\selflink(K,\lambda)+\writhe(\beta_1)-(p-1)\\
 -2p\tau_{\alpha}(K)-p(p-1)\selflink(K,\lambda)-\writhe(\beta_1)-(p-1)
 \end{array}\right\}
\leq 2\max_{\alpha}|\tau_{\alpha}(P_{\beta_1}(K))|
\end{align*}
Substituting $\omega(\beta_1)=\omega(\beta)+|\beta|-1$ and writing in terms of $c$ we have:
\begin{align*}
\max_{\alpha}\left\{ \begin{array}{r}
2p\tau_{\alpha}(K)+c-(p-1)+|\beta|-1\\
 -2p\tau_{\alpha}(K)-c-(p-1)-|\beta|+1
 \end{array}\right\}
\leq 2\max_{\alpha}|\tau_{\alpha}(P_{\beta_1}(K))|
\end{align*}
Doing the same calculation for $P_{\beta_2}(K)$, we obtain:
\begin{align*}
\max_{\gamma}\left\{ \begin{array}{r}
2p\tau_{\gamma}(K)+c-(p-1)-|\beta|+1\\
 -2p\tau_{\gamma}(K)-c-(p-1)+|\beta|-1
 \end{array}\right\}
\leq 2\max_{\gamma}|\tau_{\gamma}(P_{\beta_2}(K))|
\end{align*}

Putting this all together, we have $2\max_{\alpha,\gamma}\{|\tau_{\alpha}(P_{\beta_1}(K))|,|\tau_{\gamma}(P_{\beta_2}(K))|\}$ is bounded below by:
\begin{align*}
\max_{\alpha,\gamma}\left\{ \begin{array}{rr}
2p\tau_{\alpha}(K)+c-(p-1)+|\beta|-1 &2p\tau_{\gamma}(K)+c-(p-1)-|\beta|+1\\
 -2p\tau_{\alpha}(K)-c-(p-1)-|\beta|+1&-2p\tau_{\gamma}(K)-c-(p-1)+|\beta|-1
 \end{array}\right\}
\end{align*}
Thus,
\begin{align*}
\max_{\alpha,\gamma}\left\{ \begin{array}{r}
2p\tau_{\alpha}(K)+c-(p-1)+|\beta|-1\\
-2p\tau_{\gamma}(K)-c-(p-1)+|\beta|-1
 \end{array}\right\}
 \leq 2\max_{\alpha,\gamma}\{|\tau_{\alpha}(P_{\beta_1}(K))|,|\tau_{\gamma}(P_{\beta_2}(K))|\}.
\end{align*}
\end{proof}
Combining Lemma \ref{lem:Boundmaxmin} with Equation (\ref{eqn:Betai}) we have: 
\begin{align*} 
\max_{\alpha,\gamma}\left\{ \begin{array}{r}
2p\tau_{\alpha}(K)+c-(p-1)+|\beta|-1\\
-2p\tau_{\gamma}(K)-c-(p-1)+|\beta|-1
 \end{array}\right\}
 \leq -\chi(\Sigma)+|\beta|.
\end{align*}
Thus, for a fixed choice of rational slice surface:
\begin{align}\label{eqn:SurfaceBound}
\max_{\alpha,\gamma}\left\{ \begin{array}{r}
2\tau_{\alpha}(K)+\frac{c}{p}\\
-2\tau_{\gamma}(K)-\frac{c}{p}
 \end{array}\right\}
 \leq \frac{-\chi(\Sigma)+p}{p}.
\end{align}

To finish, we take the infimum of both sides over all $p$ and $\Sigma$. On the left hand side, this means finding the infimum over all choices of $c/p$. Letting $c/p$ vary, the smallest value of the left hand side occurs when: $$2\tau_{\alpha}(K)+\frac{c}{p}=-2\tau_{\gamma}(K)-\frac{c}{p}.$$ Solving for $c/p$ and plugging the result into Equation \ref{eqn:SurfaceBound} shows: $$\tau_{\max}(K)-\tau_{\min}(K)=\max_{\alpha,\gamma}(\tau_{\alpha}(K)-\tau_{\gamma}(K))\leq \inf_{\Sigma,p}\frac{-\chi(\Sigma)+p}{p}=2\lVert K\rVert_{Y\times [0,1]} +1.$$  
\end{proof}

\begin{remark}\label{rmk:cValue}
     One should interpret Theorem \ref{thm:RationalGenusBound} as a universal lower bound for an infinite number of minimal genus problems. Within these, one could interpret the value of $c$ as a sort of boundary condition for rational slice surfaces. It is always integral and, varying over all rational slice surfaces, can take on any integer value. In certain instances, it makes sense to consider rational slice surfaces relative to a fixed $c$-value. Equation \ref{eqn:SurfaceBound} gives a genus bound for each value of $c$, which, for a fixed $c$-value, can be stronger than the bound in Theorem \ref{thm:RationalGenusBound}. Note, however, that  there are infinitely many braids in infinitely many braid groups having any particular value of $c$, so fixing this choice does not resolve the infinitude of minimal genus problems one has to consider.

    For instance, in \cite{Wu-Yang} Wu and Yang consider rational slice surfaces whose induced satellites are a minimal parallel of the rational longitude. We will refer to such surfaces as \emph{Seifert framed} rational slice surfaces. In Corollary \ref{cor:RatLongframed} below, we show Seifert framed rational slice surfaces satisfy the boundary condition $c=0$. 
\end{remark}
\begin{restatable}{cor}{RatLongFramed}\label{cor:RatLongframed}
    If $\Sigma$ is a Seifert framed rational slice surface for a rationally null-homologous knot $K\subset Y$, then, for each nontrivial Floer class $\alpha$ in $\HFa(Y)$, $$2|\tau_{\alpha}(K)|\leq \frac{-\chi(\Sigma)}{p}+1.$$
\end{restatable}

\begin{proof}
    Choose a framing $\lambda$ for $K$. We can describe the satellite link induced by $\Sigma$ as a cable link: $$P_{\beta}(K,\lambda)= K_{mr,ms}$$ where $\gcd(r,s)=1$,  $r\lambda+s\mu$ is the homology class in $\partial \nu(K)$ of the rational longitude, and $\beta$ is the standard $mr$-braid representative of the $(mr,ms)$ torus link. Therefore, $\selflink(K,\lambda)= -\frac{s}{r}$ and $\writhe(\beta)=(mr-1)ms$. The value of $c$ is: $$c=mr(mr-1)\cdot -\frac{s}{r}+ (mr-1)ms=0.$$

    Applying Equation \ref{eqn:SurfaceBound} to any rational slice surface satisfying the boundary condition $c=0$ gives: 
\begin{align*}
    \max_{\alpha}2|\tau_{\alpha}(K)|\leq \max_{\alpha,\gamma}\left\{ \begin{array}{r}
2\tau_{\alpha}(K)\\
-2\tau_{\gamma}(K)
 \end{array}\right\}
 \leq \frac{-\chi(\Sigma)+p}{p}.
\end{align*}
\end{proof}
\begin{remark}
    In fact, the conclusion of Corollary \ref{cor:RatLongframed} holds for any rational slice surface satisfying the boundary condition $c=0$.  As alluded to above, this is a significantly wider class of surfaces than those which are Seifert framed. In particular, there are infinitely many distinct satellite links with each of infinitely many distinct winding numbers, all of which have $c=0$. Corollary \ref{cor:RatLongframed} gives slice genus bounds for all of these satellites.
\end{remark}

\subsection{A rational PL slice genus bound}
The difference $\tau_{\max}-\tau_{\min}$ is insensitive to connect summing with local knots: if $K\subset Y$ is rationally null-homologous and $J$ is a knot in $S^3$, connect summing $K$ with $J$ shifts \emph{every} $\tau$ invariant of $K$ by $\tau(J)$. This insensitivity to  local knotting has implications in the PL-category. In fact, the breadth of $\tau$ invariants bounds the rational PL slice genus, in addition to the rational (smooth) slice genus. Corollary \ref{cor:rationalPL} and Theorem \ref{cor:rationalPLslice} make this precise. 

A piecewise linear (PL) surface embedded in a 4-manifold is an embedded surface that is smooth except at a finite collection of singular points, whose neighborhoods are homeomorphic to cones on knots in the 3-sphere \cite[Appendix A]{Hom-Levine-Lidman}. If two knots $K_0$ and $K_1$ are PL-cobordant, we can replace the PL-cobordism between them by a smooth cobordism of the same genus between $K_0\#K_0'$ and $K_1\# K_1'$ where $K_0'$ and $K_1'$ are knots corresponding to links of the singular points of the surface.  

The $\tau$ breadth's insensitivity to connected summing with local knots immediately implies it is a PL-concordance invariant. Furthermore, if $K$ and $K'$ are a pair of homologous knots, then the difference of $\tau_{\max}(K')-\tau_{\min}(K')-(\tau_{\max}(K)-\tau_{\min}(K))$ is a lower bound for the PL cobordism genus. 
Celoria studied these ideas in the special case of lens spaces in \cite{Celoria}.

In addition:

\begin{cor}\label{cor:rationalPL}
If $\Sigma$ is a rational PL slice surface of degree $p$ over a rationally null-homologous knot $K$, then $$\tau_{\max}(K)-\tau_{\min}(K)\leq \frac{-\chi(\Sigma)+p}{p}.$$
\end{cor}

\begin{proof} Let $\Sigma$ be a rational PL slice surface for $K$ inducing the satellite $P_\beta(K)=P_\beta(K,\lambda)$ for some framing $\lambda$ and $p$-braid $\beta$. As in the proof of Theorem \ref{thm:RationalGenusBound}, if $P_\beta(K)$ is a link, form two related knots $P_{\beta_1}(K)$ and $P_{\beta_2}(K)$ by introducing $|\beta|-1$ positive or negative crossings, respectively. Accordingly, build PL surfaces $\Sigma_1$ and $\Sigma_2$ by attaching $|\beta|-1$ positive, respectively negatively, twisted bands. 

For each singular point $x_i$ on $\Sigma_1$ choose an arc $\gamma_i$ from $x_i$ to $P_{\beta_1}(K)$. Removing these arcs, $\Sigma \setminus \bigcup_i\nu(\gamma_i)$ is a smooth surface with boundary $P_{\beta_1}(K)\# J_1\#\ldots \#J_n$, and we can isotope the boundary of the resulting surface $\Sigma'$ along the $\gamma_i$'s so that $P_{\beta_1}(K)\# J_1\#\ldots \#J_n$ is contained in $Y\times\{1-\epsilon\}$. Then 
\begin{align*}
    |\tau_\alpha(P_{\beta_1}(K)\# J_1\#\ldots \#J_n)|=|\tau_{\alpha}(P_{\beta_1}(K))+\tau(J_1)+\ldots \tau(J_n)|&\leq -\chi(\Sigma_1)+1\\
    &\leq -\chi(\Sigma)+|\beta|.
\end{align*}    The same construction for $\Sigma_2$ shows: $$|\tau_{\alpha}(P_{\beta_2}(K))+\tau(J_1)+\ldots \tau(J_n)|\leq -\chi(\Sigma)+|\beta|.$$
Following through the proof of Theorem \ref{thm:RationalGenusBound} until Equation \ref{eqn:SurfaceBound} now shows:
\begin{align*}
\max_{\alpha,\gamma}\left\{ \begin{array}{r}
2\tau_{\alpha}(K)+\frac{c}{p}+\frac{1}{p}(\tau(J_1)+\ldots +\tau(J_n))\\
-2\tau_{\gamma}(K)-\frac{c}{p}-\frac{1}{p}(\tau(J_1)-\ldots- \tau(J_n))
 \end{array}\right\}
 \leq \frac{-\chi(\Sigma)+p}{p}.
\end{align*}
The same argument as in the conclusion of Theorem \ref{thm:RationalGenusBound} now implies the corollary.
\end{proof}
\noindent Corollary \ref{cor:rationalPL} immediately implies Theorem
 \ref{cor:rationalPLslice}.

\section{Applications}
\subsection{Floer simple knots} A Floer simple knot is one whose knot  Floer homology is minimal, in the sense that: $$\HFKa(Y,K)\cong \HFa(Y).$$  Thus, Floer simple knots have the property that any non-zero class  $\alpha\in \HFKa_a(Y,K)$ in Alexander grading $a$  survives the spectral sequence to $\HFa(Y)$, implying $\tau_{\alpha}(K)=a$. It follows that the breadth of the $\tau$ invariants agrees with the breadth of knot Floer homology  $$\tau_{\max}(K)-\tau_{\min}(K)=A_{\max}(K)-A_{\min}(K).$$
Theorem \ref{thm:RationalGenusBound} now implies that the rational Seifert and slice genera agree for Floer simple knots (Theorem \ref{thm:FSK} from the introduction, restated here).

\FSK*
\begin{proof} 
Theorem \ref{thm:RationalGenusBound} shows that $\tau_{\max}(K)-\tau_{\min}(K)\leq 2\lVert K\rVert_{Y\times [0,1]}+1$. On the other hand, since $K$ is Floer simple, $2\lVert K\rVert_Y+1=A_{\max}(K)-A_{\min}(K)= \tau_{\max}(K)-\tau_{\min}(K)$. Equation \ref{eqn:4genus3genus} now implies the result.
    \end{proof}

 In light of Corollary \ref{cor:RatLongframed} and the above discussion for Floer simple knots, it is natural to ask whether $\max_{\alpha}|2\tau_{\alpha}(K)|$ is a lower bound for $2\lVert K\rVert_{Y\times [0,1]}+1$, which would be better than the bound in Theorem \ref{thm:RationalGenusBound}. The following example shows this is not the case. This example also demonstrates that the rational slice genus differs from the rational slice genus relative to the rational longitude, as considered by Wu and Yang \cite{Wu-Yang}.

\begin{example}\label{ex:SmallGenusCable}
Let $J=(\RP3,\RP1)\#(S^3,T_{2,-5})$ be  the connected sum of the  knot $\RP1
\subset \RP3$  with the negative torus knot $T_{2,-5}$ in the 3-sphere. Then $J$ is a knot of order 2 in $\RP3$. 

We use additivity to compute the $\tau$ invariants of $J$. First, since $\RP1\subset \RP3$ is a Floer simple knot in an L-space, there is a single $\tau$ invariant in each $\SpinC$ structure and $2\tau_\spinc(\RP1)=d_\spinc(\RP3)-d_{\spinc+\PD[\RP1]}(\RP3)$. The $d$-invariants of $\RP3$ are $\pm\frac{1}{4}$ and $\PD[\RP1]$ is order 2. Thus, $\tau_{\max}(\RP1)=\frac{1}{4}$ and $\tau_{\min}(\RP1)=-\frac{1}{4}$. One can calculate directly that $\tau(T_{2,-5})=-2$. Thus: \begin{align*}
    \tau_{\max}(J)&=\frac{1}{4}-2=-7/4 &\tau_{\min}(J)&=-\frac{1}{4}-2=-9/4.
\end{align*}

On the other hand, we can explicitly construct a rational slice surface for $J$ by taking any slicing surface for \emph{any 2-cable} of $T_{2,-5}$ in the 4-ball and band summing with a slicing disk $D$ for the dual knot. Crucially, for the resulting surface to be a rational slicing surface for $J$, we must use two bands, added along the arc used to form the connected sum. 

Specifically, let $S_2$ be the genus-2 slicing surface for the positive (2,5)-cable of $T_{2,-5}$ constructed by Hom, Lidman and Park in the proof of \cite[Corollary 4.5]{Hom-Lidman-Park}. Note that this surface demonstrates that, even in the $3$-sphere, the rational slice genus isn't determined by the slice genus. 

After attaching two bands and a disk to $S_2$, the resulting rational slicing surface $S$ for $J$ has negative Euler characteristic: $$-\chi(S)=-\chi(S_2)-\chi (D)+2=-\chi(S_2)+1=2g(S_2)=4.$$ Thus, $$2\lVert J\rVert_{\RP3\times [0,1]}+1\leq \frac{-\chi(S)+2}{2}=6/2=3.$$ This is clearly smaller than $\max_{\alpha}|2\tau_{\alpha}(J)|=9/2$.
\bigskip

In fact, Hom, Lidman and Park use the $\nu^+$ concordance invariant, which is closely related to $\tau$, to show that the (usual) slice genus of the $(2,5)$-cable of $T_{2,-5}$ is exactly 2. Thus, it is an example of a braided satellite where the slice genus of the satellite equals the slice genus of the pattern knot. It would be interesting to know whether the slice genus of a cable, or any other braided satellite, knot can drop lower below that of its companion.
    \begin{ques}
    Is the slice genus of a braided satellite knot $P_\beta(K)$ bounded below by the slice genus of its companion; i.e., is $g_4(K)\leq g_4(P_\beta(K))$?
\end{ques}

\end{example}

\subsection{Knots with large PL slice genus}
Using Theorem \ref{cor:rationalPLslice}, we show there are sequences of knots $J_i$ with arbitrarily large PL slice genus by showing that $\tau_{\max}(J_i)-\tau_{\min}(J_i)$ grows arbitrarily large. We need a proposition:
\begin{prop}\label{prop:taudifferenceboundbraidindex}
If $J\subset Y$ is a rationally null-homologous knot, then: $$p(\tau_{\max}(J)-\tau_{\min}(J)-1)+1\leq \tau_{\max}(P_\beta(J))-\tau_{\min}(P_{\beta}(J))$$ where  $P_\beta(J)$ is any braided satellite whose pattern is a $p$-braid that closes to a knot. 
\end{prop}
\begin{proof} Theorem \ref{thm:BPBounds} implies for all nontrivial Floer classes $\alpha$ and $\alpha'$ in $\HFa(Y)$:
\begin{align*}
2p\tau_{\alpha}(J)-(p-1)(p\selflink(J,\lambda)+1)+\omega(\beta)&\leq 2\tau_{\alpha}(P_\beta(J,\lambda))\\
-2p\tau_{\alpha'}(J)-(p-1)(1-p\selflink(J,\lambda))-\omega(\beta)&\leq -2\tau_{\alpha'}(P_\beta(J,\lambda))
\end{align*}
where $\lambda$ denotes any choice of framing.
Adding together, for any $\alpha$ and $\alpha'$ we have:
\begin{align*}
    p\tau_{\alpha}(J)-p\tau_{\alpha'}(J)-(p-1)&\leq \tau_{\alpha}(P_\beta(J,\lambda))-\tau_{\alpha'}(P_\beta(J,\lambda))\\&\le \tau_{\max}(P_\beta(J,\lambda))-\tau_{\min}(P_\beta(J,\lambda)).
\end{align*}
Choosing $\alpha$ and $\alpha'$ to maximize the left hand size now gives the result.
\end{proof}

Consequently, if $\tau_{\max}(J)-\tau_{\min}(J)>1$ then the left hand side grows with $p$. In addition, if $\beta$ is a $p$-braid the homology class of $P_\beta(J)$ is $p[J]$ in $H_1(Y;\Z)$. Letting $p_i$ be a sequence of integers congruent to 1 modulo the order of $J$ and $\beta_i$ a $p_i$-braid for each $i$, the satellites $J_i=P_{\beta_i}(J)$ form a sequence of knots \emph{homologous to $J$} whose rational slice genus grows arbitrarily large. Furthermore, letting $m_i$ be a sequence of integers congruent to $m$ modulo the order of $J$ and $\theta_i$ an $m_i$-braid, the satellites $P_{\theta_i}(J)$ generate sequences of examples in other homology classes. 

Corollary \ref{cor:reproof} provides an example of knots for which we can employ this idea.  We restate the corollary here:

\reproof*

\begin{proof}

    Let $K$ be an $L$-space knot of genus $g$, and $\mu$ be its meridian, viewed as a knot in $S^3_{-1}(K)$. Note that the genus of $\mu$ is also $g$, as the two knots have homeomorphic complements. The main result of \cite{Hedden-Levine} is a formula for the filtered homotopy type of the $\Z\oplus\Z$-filtered complex $\CFKinf(S^3_{-1}(K),\mu)$ in terms of a filtered mapping cone of filtered complexes derived from $\CFKinf(K)$.  Specializing that result to the case of the $\Z$-filtered complex $\CFKa(S^3_{-1}(K),\mu)$ from which the $\tau$ invariants are derived, we can easily show that there exist Floer classes $\alpha,\gamma\in \HFa(S^3_{-1}(K))$ for which the  associated  $\tau$ invariants satisfy $\tau_\alpha(\mu)=g-1$ and $\tau_\gamma(\mu)=-g$, hence
    \[ \tau_\alpha(\mu)-\tau_\gamma(\mu)=2g-1\ge 2,\]
    provided that the genus of $K$ is greater than one.  This applies to all $L$-space knots except the right-handed trefoil.  Applying Proposition \ref{prop:taudifferenceboundbraidindex} with braids of  increasing index, we find that for any $i$ there exists a braided satellite of $\mu$, which we denote $J_i=P_{\beta_i}(\mu)$ for whom the difference $\tau_{\max}-\tau_{\min}$ is bigger than $i$.  Together with Theorem \ref{cor:rationalPLslice}, we see that the rational PL slice genus, and hence the PL slice genus, grows linearly with the braid index.

    Rather than go into the details of the filtered mapping cone formula from \cite{Hedden-Levine}, we appeal to direct results from the literature which imply the   existence of classes $\alpha,\gamma$ for which $\tau_\alpha(\mu)-\tau_\gamma(\mu)\ge 2$.  For this, we note that \cite[Propositions 2.6 and 2.7]{Hedden-Berge} together directly imply that \begin{equation}\label{eq:rankdualknot}
    \rank\ \HFKa(S^3_{-1}(K),\mu)= \rank\ \HFa(S^3_{-1}(K))+2= 4g+1. \end{equation}
      Now the top and bottommost non-trivial knot Floer homology groups of $\mu$ are supported in Alexander gradings $\pm g$ by \cite[Theorem 1.2] {GenusBounds}.  By \cite{Lens,NiFibered}, the L-space knot $K\subset S^3$ is fibered. Hence $\mu\subset S^3_{-1}(K)$ is fibered as well, and by \cite[Theorem 1.1]{BVV}, the knot Floer homology of any fibered knot is non-trivial in degree $g-1$.  By conjugation symmetry, \cite[Proposition 3.10]{Knots}, the Floer homology is also non-trivial in degree $-(g-1)$.  Now the Euler characteristic of $\HFKa(S^3_{-1}(K),\mu)$, summed over all Alexander gradings, equals that of $\HFa(S^3_{-1}(K))$ since there is a spectral sequence relating them.  The latter has Euler characteristic equal one by \cite[Proposition 5.1]{HolDiskTwo}.  It follows that the knot Floer homology of $\mu$ must be non-trivial in Alexander grading zero, for if it were not then conjugation symmetry would imply that it is even.  Thus we see that there are at least 5 distinct Alexander gradings supporting non-trivial Floer homology, but Equation \eqref{eq:rankdualknot} implies that the total rank of the differentials in the spectral sequence from $\HFKa(S^3_{-1}(K),\mu)$ to $\HFa(S^3_{-1}(K))$ is one.  Thus at least three  of the five non-trivial Alexander gradings $\pm g, \pm(g-1), 0$ must contain classes which survive the spectral sequence, and  which correspondingly produce $\tau$-invariants with those values.  Any three choices among the five values yield $\tau$-invariants that differ by at least two, verifying the assertion. \end{proof}

     To obtain slice genus bounds in the homology balls  discussed in the introduction, note that if $\Sigma\subset S^3_{-1}(K)\times[0,1]$ is a PL surface with boundary $J_i\subset S^3_{-1}(K)\times \{1\}$, then up to proper isotopy we can assume that $\Sigma$ is disjoint from a neighborhood of a chosen arc connecting the boundary components of $S^3_{-1}(K)\times[0,1]$.  Removing this neighborhood, we obtain a surface embedded in the homology ball $W$ of the same genus, hence any bound for the (rational) PL slice genus of $J_i\subset S^3_{-1}(K)$ is a bound for the minimum genus of any PL surface bounded by $J_i\subset S^3_{-1}(K)\#-S^3_{-1}(K)$ in $W$.

\subsection{Deep slice knots in double branched covers}
 A knot $K\subset \partial X^4$ is slice in a compact 4-manifold $X^4$ if it bounds a properly embedded disk. A slice knot is \emph{shallow slice} if it bounds a disk in $\partial X\times I\subset X$. Otherwise we say $K$ is \emph{deep slice} \cite{Klug-Ruppik}.  

In \cite[Corollary 5.2]{RelativeAdjunction}, the authors show each $\tau$ invariant of a null-homologous knot $K\subset Y$ bounds the genus of a surface embedded in $Y\times [0,1]$ with boundary $K$: $$|\tau_{\alpha}(K)|\leq g_{Y\times [0,1]}(K).$$ Thus, if $K$ is shallow slice \emph{every} $\tau_{\alpha}$ invariant must vanish.

Let $K$ be a slice knot in $S^3=\partial B^4$. Taking the double branched cover of $B^4$ along a slice disk, we obtain a rational homology ball whose boundary $\Sigma(K)$ is the double branched cover of $S^3$ along $K$. The knot $K$ lifts to a null-homologous knot $\widetilde{K}$ in $\Sigma(K)$. Moreover, $\widetilde{K}$ is slice, since the slice disk in $B^4$ lifts to a disk in the rational ball.  We prove  Proposition \ref{prop:DeepSlice} from the introduction:

\DeepSlice*

\begin{proof}   
In \cite{ALevine-branchedcovers}, Levine computed the knot Floer homology of the lifts of 3-bridge knots in their double branched covers. The knots $8_{20}, 10_{129}$ and $10_{140}$ are knots on Levine's list which are simultaneously slice and Khovanov-thin \cite{knotinfo}. We will use the fact that double branched covers of Khovanov-thin knots are Heegaard Floer $L$-spaces to help us calculate their $\tau$ invariants.

Levine records the knot Floer homology in the form of a Poincar\'e polynomial $p_\spinc(q,t)$ where the $q$ exponent records the Maslov grading and the $t$ exponent the Alexander grading. He labels the unique $\SpinC$ structure corresponding to a spin structure by $0$ and the others accordingly, so that conjugate $\SpinC$ structures have opposite sign.

We claim $\tau_{\pm1}(\widetilde{8_{20}})\neq 0$. 
Levine computes the knot Floer homology of $\widetilde{8_{20}}$ in the relevant $\SpinC$ structures as: $$p_{\pm1}(q,t)=q^{7/9}(q^{-1}t^{-1}+1+qt).$$ There is a single generator in each supported Alexander grading and Maslov grading. Since $\Sigma(8_{20})$ is an $L$-space, a single generator survives in the spectral sequence from $\HFKa$ to generate $\HFa(\Sigma(8_{20}),\pm1)$. Moreover, the generators that do not survive differ in Maslov grading by 1, since the differential lowers this  grading by 1. It follows that $\tau_{\pm1}(K)=1$, or $-1$. Thus, the lift of $8_{20}$ is deep slice. 

Analogous calculations show the lifts of $10_{129}$ and $10_{140}$ are also deep slice.
\end{proof}

\subsection{A rational slice-Bennequin Inequality}\label{sec:sliceTB} Let $(Y,\xi)$ be a contact manifold and $\mathcal{K}$ a rationally null-homologous Legendrian. If we restrict to considering rational slice surfaces which are Seifert framed, then we have a well-defined notion of the rational Thurston-Bennequin number. The contact structure determines a framing $\lambda_{\tb}$ for $\mathcal{K}$. The rational Thurston-Bennequin number as defined by Baker and Etnyre in \cite{Baker-Etnyre} is $$\tb_{\Q}(\mathcal{K}):=\selflink(\mathcal{K},\lambda_{\tb})=\frac{1}{p} F\cdot \lambda_{\tb}$$ where $F$ is a rational Seifert surface for $\mathcal{K}$ of degree $p$. This linking number is also given by $\frac{1}{p^2}F\cdot F'$ where $F$ is now considered as a pushed-in rational Seifert surface in $Y\times [0,1]$ and $F'$ is a pushoff determined by the framing $\lambda_{\tb}$, see \cite[Section 1]{Baker-Etnyre}. For a general rational slice surface, $\Sigma\cdot\Sigma'\neq F\cdot F'$. However, if $\Sigma$ is Seifert framed, its self-intersection will agree with that of a push-in of a rational Seifert surface of the same degree over $K$. 

Restricting to Seifert framed rational slice surfaces, $\tb_{\Q}$ is therefore well-defined and we can prove a rational slice Bennequin inequality. To state it, we first recall Baker and Etnyre's 3-dimensional rational Bennequin inequality. In addition to defining $\tb_{\Q}(\mathcal{K})$, discussed above, they define the rational rotation number: $\rot_{\Q,[F]}(\mathcal{K})$ and prove that if $\mathcal{K}$ is a rationally null-homologous Legendrian in a tight contact 3-manifold, then \[\tb_\Q(\mathcal{K})+\rot_{\Q,[F]}(\mathcal{K})\le -\frac{1}{p}\chi(F),\] where $F$ is a rational Seifert surface for $\mathcal{K}$, and $p$ is the degree of the cover $\partial F\to K$ \cite{Baker-Etnyre}.  We now prove the 4-dimensional analogue stated in the introduction, which holds whenever the Ozsv\'ath-Szab\'o contact class is non-trivial:

\rationalslicebennequin*

\begin{proof} The proof is similar in spirit to Theorem 1 of \cite{4Dtight}. Let $K$ denote the underlying topological knot type of $\mathcal{K}$. In \cite{Contact}, the first author defines an invariant $\tau_\xi(K)$ of null-homologous knots. Work of both authors shows this construction naturally extends to rationally null-homologous knots \cite{RelativeAdjunction}. In particular, we have $\tau_\xi(Y,K)=-\tau_{c(\xi)}(-Y,K)$ where $c(\xi)$ denotes the Ozsv\'ath-Szab\'o contact class. 

In \cite[Theorem 1.1]{Li-Wu}, Li and Wu show if the contact class is nontrivial, then \[\tb_\Q(\mathcal{K})+\rot_{\Q,[F]}(\mathcal{K})\le 2\tau_\xi(K)-1\] where $\tb_{\Q}$ and $\rot_{\Q,[F]}$ are both defined with respect to a rational Seifert surface $F$. 

If $\Sigma$ is a Seifert framed rational slice surface representing the same homology class as $F$ in the knot exterior, then $\rot_{\Q,[\Sigma]}(\mathcal{K})=\rot_{\Q,[F]}(\mathcal{K})$. Combined with Corollary \ref{cor:RatLongframed} we have:
\[\tb_\Q(\mathcal{K})+\rot_{\Q,[\Sigma]}(\mathcal{K})\le 2\tau_\xi(Y,K)-1=-2\tau_{c(\xi)}(-Y,K)-1\le -\frac{1}{p}\chi(\Sigma).\]
\end{proof}

\subsection{A 4-dimensional Turaev function} \label{sec:RatGenus}
We end with several questions about the relationship of rational 4-genus to  the  rational extension of the Thurston norm introduced by Turaev. In \cite{Turaev-function}, Turaev introduces a function: $\Theta_Y: \Tors (H_1(Y;\Z))\to \R$. For a torsion homology class $k$ in $H_1(Y;\Z)$ we have: $$\Theta_{Y}(k):=\inf_{\substack{K\subset Y \\ [K]=k}}2||K||_Y.$$ 

\noindent Replacing the rational Seifert genus $\lVert K\rVert_Y$ with the rational slice genus $\lVert K\rVert_{Y\times [0,1]}$ we define a 4-dimensional analogue of Turaev's function: $$\Theta_{Y\times [0,1]}(k):=\inf_{\substack{K\subset Y \\ [K]=k}} 2||K||_{Y\times [0,1]}.$$

\noindent Clearly there is an inequality $\Theta_{Y\times [0,1]}(k)\leq \Theta_{Y}(k)$, since a rational Seifert surface always gives rise to a rational slice surface by pushing in. We find the following question compelling:

\begin{ques}
For which manifolds $Y$ and homology classes $k$ do we have equality: $\Theta_{Y\times [0,1]}(k) = \Theta_{Y}(k)$?
\end{ques}

Theorem \ref{thm:RationalGenusBound} implies that the infimum of $\tau_{\max}-\tau_{\min}$ taken over all knots in the same homology class is a lower bound for the 4-dimensional Turaev function: $$\inf_{\substack{K\subset Y \\ [K]=k}} \{\tau_{\max}(K)-\tau_{\min}(K)\}\leq \Theta_{Y\times [0,1]}(k) +1.$$

\noindent Floer simple knots seem particularly interesting in this context. Ni and Wu showed these knots always achieve $\Theta_Y$ \cite{Ni-Wu}. In light of Theorem \ref{thm:FSK}, we would like to know if Floer simple knots also realize $\Theta_{Y\times [0,1]}$.

\begin{ques}\label{ques:NormMinimize} If a class $k$ in $H_1(Y,\Z)$ can be represented by a Floer simple knot $K$ does $K$ realize $\Theta_{Y\times [0,1]}(k)$? In other words, is: $$\Theta_{Y\times [0,1]}(k)=2\lVert K\rVert_{Y}=2\lVert K\rVert_{Y\times [0,1]}?$$
\end{ques}

\noindent Wu and Yang show that if we instead consider the rational genus relative to the rational longitude in the definition of all three quantities above, then the equalities hold \cite{Wu-Yang}. There is some evidence that the answer to Question \ref{ques:NormMinimize} is positive even with the present stronger notions of rational slice genus.  Indeed, focusing on homology classes with $\Z/2\Z$ coefficients, work of Levine, Ruberman and Strle implies the answer to Question \ref{ques:NormMinimize} is yes for homology classes of order 2 \cite[Theorem 7.2]{Levine-Ruberman-Strle}. Specifically, they prove that in rational homology spheres, $\Theta_{Y\times [0,1]}(k)$ is bounded below by differences of Heegaard Floer $d$-invariants: $$\max_{\spinc\in\SpinC(Y)}\{d(Y,\spinc+k)-d(Y,\spinc)\}\leq \Theta_{Y\times [0,1]}(k)+1.$$ Combining this with Ni and Wu's theorem \cite{Ni-Wu}, they show this bound is sharp when  $Y$ is an $L$-space and the homology class $k$ contains a Floer simple knot. 

The following proposition establishes a relationship between the infimum of the $\tau$ breadth over all knots representing a fixed homology class $k$ and the maximum difference of $d$-invariants:
  
\begin{prop} If $K$ is a knot in a rational homology sphere whose homology class is represented by a Floer simple knot, then
$$\inf_{\substack{K\subset Y \\ [K]=k}} \tau_{\max}(K)-\tau_{\min}(K)\leq \max_{\spinc\in\SpinC(Y)}\{d(Y,\spinc+k)-d(Y,\spinc)\}.$$
\end{prop}
\begin{proof}
For any knot we have the following string of inequalities: $$A_{\min}(K)\leq \tau_{\min}(K)\leq \tau_{\max}(K)\leq A_{\max}(K).$$ Thus $\tau_{\max}(K)-\tau_{\min}(K)\leq A_{\max}(K)-A_{\min}(K)$. Minimizing over all knots that are homologous to $K$ we have: 
$$\inf_{\substack{K\subset Y \\ [K]=k}} \tau_{\max}(K)-\tau_{\min}(K)\leq\min_{\substack{K\subset Y \\ [K]=k}} A_{\max}(K)-A_{\min}(K)= \max_{\spinc\in\SpinC(Y)}\{d(Y,\spinc+k)-d(Y,\spinc)\},$$ where the final equality follows from Ni and Wu \cite{Ni-Wu}.
\end{proof}

This motivates our final question:

\begin{ques} If $k$ is represented by a Floer simple knot in an L-space, do we  have equality: $$\max_{\spinc\in\SpinC(Y)}\{d(Y,\spinc+k)-d(Y,\spinc)\}= \inf_{\substack{K\subset Y \\ [K]=k}} \tau_{\max}(K)-\tau_{\min}(K)?$$
Or, are there knots representing $k$ with smaller $\tau_{\max}-\tau_{\min}$ than the Floer simple knot?
\end{ques}

\bibliographystyle{plain}
\bibliography{mybib}
\end{document}

%% file: macros-MTH-Hedden.tex
\newcommand{\Tors}{\mathrm{Tors}}

\newtheorem{thm}{Theorem}

\newtheorem*{Adjunction Inequality}{Adjunction Inequality}

\newtheorem{theorem}{Theorem}[section]
\newtheorem{prop}[theorem]{Proposition}

\newtheorem{cor}[theorem]{Corollary}

\newtheorem{lemma}[theorem]{Lemma}

\newtheorem{question}[thm]{Question}
\newtheorem{ques}[theorem]{Question}
\theoremstyle{definition}
\newtheorem{example}[theorem]{Example}
\newtheorem{defn}[theorem]{Definition}
\newtheorem{remark}[theorem]{Remark}

\def\endproofof{\relax\ifmmode\expandafter\endproofmath\else
  \unskip\nobreak\hfil\penalty50\hskip.75em\hbox{}\nobreak\hfil\bull
  {\parfillskip=0pt \finalhyphendemerits=0 \bigbreak}\fi}
\def\endproofofmath$${\eqno\bull$$\bigbreak}

\def\endproof{\relax\ifmmode\expandafter\endproofmath\else
  \unskip\nobreak\hfil\penalty50\hskip.75em\hbox{}\nobreak\hfil\bull
  {\parfillskip=0pt \finalhyphendemerits=0 \bigbreak}\fi}
\def\endproofmath$${\eqno\bull$$\bigbreak}
\def\bull{\vbox{\hrule\hbox{\vrule\kern3pt\vbox{\kern6pt}\kern3pt\vrule}\hrule}}

\newcommand{\Q}{\mathbb{Q}}
\newcommand{\R}{\mathbb{R}}

\newcommand{\Z}{\mathbb{Z}}

\newcommand\RP[1]{{\mathbb{RP}}^{#1}}

\newcommand\SpinC{\mathrm{Spin}^c}
\newcommand{\F}{\mathbb F}

\newcommand\relspinc{\underline{\spinc}}

\newcommand\Filt{\mathcal F}

\newcommand\x{\mathbf x}

\newcommand\z{\mathbf z}

\newcommand\y{\mathbf y}

\newcommand\sP{\mathcal{P}}

\newcommand\ModSphere{\ModFlow\left({\mathbb S}\longrightarrow 
\Sym^{g-1}(\Sigma_{1})\times \Sym^2(\Sigma_{2})\right)}
\newcommand\ModSpheres\ModSphere

\newcommand\CFa{\widehat{CF}}

\newcommand\HFa{\smash{\widehat{HF}}}

\newcommand\UnparModSp{\widehat \ModSp}
\newcommand\UnparModFlow\UnparModSp
\newcommand\Mod\ModSp

\newcommand\PD{\mathrm{PD}}

\newcommand{\spinc}{\mathfrak s}

\newcommand{\spinct}{\mathfrak t}

\newcommand\ModMaps{\mathcal M}
\newcommand\ModSp\ModMaps

\newcommand\Ta{{\mathbb T}_{\alpha}}
\newcommand\Tb{{\mathbb T}_{\beta}}

\newcommand\alphas{\mbox{\boldmath$\alpha$}}

\newcommand\betas{\mbox{\boldmath$\beta$}}

\newcommand\spincrel\relspinc

\newcommand\CFK{CFK}
\newcommand\HFK{HFK}

\newcommand\CFKa{\widehat\CFK}

\newcommand\CFKinf{\CFK^{\infty}}

\newcommand\HFKa{\widehat\HFK}

\newcommand\Dual{\mathcal D}
\newcommand\Duality\Dual

\newcommand\Tor{\mathrm{Tor}}

\newcommand\Sym{\mathrm{Sym}}

\newcommand\SoneStwo{S^1\times S^2}

\newcommand{\bignatural}{\mathop{\vcenter{\hbox{\Large$\natural$}}}}

\newcommand{\rank}{\mathrm{rank}}

\newcommand\maxa{\mathfrak{a}}
\newcommand\tb{\operatorname{tb}}
\newcommand\rot{\operatorname{rot}}

\newcommand\selflink{\lk_{\Q}}
\newcommand\qzselflink{\operatorname{sl}_{\Q\slash\Z}}
\newcommand\nlink{\lk_{\Q}(K,\lambda_n)}
\newcommand\Nlink{\lk_{\Q}(K,\lambda_N)}
\newcommand\writhe{\omega}
\newcommand\len{\operatorname{len}}
\newcommand\lk{\operatorname{lk}}